\documentclass[reqno, 11pt]{amsart}
\usepackage[left=2cm, right=2cm, top=2cm, bottom=2cm]{geometry}
\usepackage[T1]{fontenc}
\usepackage[utf8]{inputenc}	
\usepackage[english]{babel}
\usepackage[babel]{csquotes}
\usepackage{amssymb}
\usepackage{amsmath}
\usepackage{amsthm}
\usepackage{latexsym}
\usepackage{xcolor}
\usepackage{mathrsfs}
\usepackage{bbm}
\usepackage{verbatim}
\usepackage[colorlinks=true, pdfstartview=FitV, linkcolor=blue, citecolor=blue, urlcolor=blue]{hyperref}

\usepackage{faktor}
\usepackage[shortlabels]{enumitem}
\usepackage[english]{babel}

\newcommand{\cF}{{\cal F}}

\newcommand{\EE}{{\mathbb E}}

  \newcommand{\PX}{{\mathbb P}}
\newcommand{\PP}{{\mathbb P}}

\renewcommand{\cF}{\mathcal F}

\numberwithin{equation}{section}
\newtheorem{theorem}{Theorem}[section]

\newtheorem{lemma}[theorem]{Lemma}
\newtheorem{remark}[theorem]{Remark}
\newtheorem{prop}[theorem]{Proposition}
\newtheorem{coro}[theorem]{Corollary}

\newtheorem{Assumption}[theorem]{Assumption}

\makeatletter
\DeclareFontFamily{OMX}{MnSymbolE}{}
\DeclareSymbolFont{MnLargeSymbols}{OMX}{MnSymbolE}{m}{n}
\SetSymbolFont{MnLargeSymbols}{bold}{OMX}{MnSymbolE}{b}{n}
\DeclareFontShape{OMX}{MnSymbolE}{m}{n}{
	<-6>  MnSymbolE5
	<6-7>  MnSymbolE6
	<7-8>  MnSymbolE7
	<8-9>  MnSymbolE8
	<9-10> MnSymbolE9
	<10-12> MnSymbolE10
	<12->   MnSymbolE12
}{}
\DeclareFontShape{OMX}{MnSymbolE}{b}{n}{
	<-6>  MnSymbolE-Bold5
	<6-7>  MnSymbolE-Bold6
	<7-8>  MnSymbolE-Bold7
	<8-9>  MnSymbolE-Bold8
	<9-10> MnSymbolE-Bold9
	<10-12> MnSymbolE-Bold10
	<12->   MnSymbolE-Bold12
}{}

\let\llangle\@undefined
\let\rrangle\@undefined
\DeclareMathDelimiter{\llangle}{\mathopen}%
{MnLargeSymbols}{'164}{MnLargeSymbols}{'164}
\DeclareMathDelimiter{\rrangle}{\mathclose}%
{MnLargeSymbols}{'171}{MnLargeSymbols}{'171}
\makeatother

\usepackage[]{amsaddr}

\begin{document}
	\title[Continuous data assimilation 
    for 2D stochastic Navier--Stokes equations ]{Continuous data assimilation\\  for 2D stochastic Navier--Stokes equations}
	\author{Hakima Bessaih, Benedetta Ferrario, Oussama Landoulsi  \and Margherita Zanella} 
    \address{\small{Department of Mathematics and Statistics, Florida International University, 11200 SW 8th St, Miami, FL 33199}}
    \address{\small{Dipartimento di Scienze Economiche e Aziendali, Universit\`a di Pavia, 27100, Pavia, Italy}} 
  \address{\small{Department of Mathematics and Statistics, University of Massachusetts Amherst, 710 N. Pleasant Street, Amherst, MA 01003-9305, USA}}
	\address{\small{Dipartimento di Matematica, Politecnico di Milano, Via E.~Bonardi 9, 20133 Milano, Italy}}
    \email{hbessaih@fiu.edu}
	\email{benedetta.ferrario@unipv.it}
  \email{olandoulsi@umass.edu}
	\email{margherita.zanella@polimi.it}
	
	\begin{abstract}
Continuous data assimilation methods, such as the nudging algorithm introduced by Azouani, Olson, and Titi (AOT) \cite{azouani2014continuous}, are known to be highly effective in deterministic settings for asymptotically synchronizing approximate solutions with observed dynamics. In this work, we extend this framework to a stochastic regime by considering the two-dimensional incompressible Navier-Stokes equations subject to either additive or multiplicative noise. We establish sufficient conditions on the nudging parameter and the spatial observation scale that guarantee convergence of the nudged solution to the true stochastic flow.

In the case of multiplicative noise, convergence holds in expectation, with exponential or polynomial rates depending on the growth of the noise covariance. For additive noise, we obtain the exponential convergence both in expectation and pathwise.
These results yield a stochastic generalization of the AOT theory, demonstrating how the interplay between random forcing, viscous dissipation and feedback control governs synchronization in stochastic fluid systems.
	\end{abstract}
	
	\maketitle

  \tableofcontents

\section{Introduction}
Accurate computational prediction of nonlinear dynamical systems, such as fluid flows governed by the Navier-Stokes equations, faces two fundamental challenges: the initialization problem and the maintenance of long-time accuracy. The initialization problem arises because the complete initial state of the system is typically unknown or only partially observed with limited precision. The long-time accuracy problem concerns the accumulation of modeling and numerical errors during the systems evolution, which may lead to divergence from the true state.

Data assimilation provides a rigorous framework to address these challenges by combining observational measurements with the governing dynamical equations to produce reliable approximations of the true system state. The classical method of continuous data assimilation (CDA), as described by Daley \cite{daley1993atmospheric}, consists of directly incorporating observational measurements into a model as it is being integrated in time.
A more recent approach, introduced by Azouani, Olson, and Titi (AOT) in \cite{azouani2014continuous}, reformulates CDA as a feedback-control (or nudging) mechanism, inspired by ideas from control theory \cite{AbderrahimTiti2014}. In this formulation, an additional relaxation term is introduced to nudge the model solution toward the observed data. The resulting system generates an approximate trajectory that converges asymptotically to the reference (true) solution. Under suitable assumptions on the data and model regularity, convergence can be shown to occur at an exponential rate. This algorithm offers a practical and efficient approach to estimating the current and future states of a dynamical system when only sparse observational data are available.
\\
The AOT algorithm has since been rigorously analyzed for a wide range of physical models, confirming its effectiveness in diverse settings. Examples include the Navier-Stokes equations with noisy data \cite{bessaih2015continuous} and \cite{Blomker}, discrete data assimilation schemes \cite{foias2016discrete}, miscible displacement in porous media \cite{bessaih2022continuous}, the 3D Navier-Stokes-$\alpha$ model \cite{albanez2016continuous}, the 2D Rayleigh--B\'enard convection problem \cite{FGHMMW, FLT, FJT}, the surface quasi-geostrophic equation \cite{JMT, JMOT}, the Korteweg de Vries equation \cite{JST}, the 2D magnetohydrodynamic system \cite{BHLP}, the 3D Brinkman--Forchheimer--extended Darcy model \cite{MTT}, the 3D primitive equations \cite{P}, the 3D Leray--$\alpha$ model \cite{FLT2}, the Voigt relaxation of the 2D Navier--Stokes equations \cite{LP} and the 3D Ladyzhenskaya model \cite{CGJA}.

\smallskip

Let us briefly recall the general structure of the continuous data assimilation algorithm. Let 
$u=u(t)$ denote the unknown true state of a dynamical system governed by the equation
\[
\frac{{\rm  d}u(t)}{{\rm d}t} = F(u(t)),\qquad t>0
\]
with an unknown initial condition $u(0)$.
Observations are available only at a coarse spatial scale through an interpolant operator  \( R_h(u(t)) \), where \( h > 0 \) 
 denotes the observation resolution. The CDA algorithm constructs an approximate solution of the modified equation
 \[
\frac{{\rm d}U(t)}{{\rm d}t} = F(U(t)) + \mu R_h(u(t) - U(t)),
\]
where  \( \mu>0 \) is the nudging parameter. Under suitable conditions on  \( \mu \) and \( h \) it can be shown that 
\( U(t) \) converges to \( u(t) \) as $t\to+ \infty$, for any  initial data $u(0)$ and $ U(0)$.
The estimates controlling the difference between $U$ and $u$  are often referred to as Foias--Prodi estimates, and they typically guarantee exponential decay in time provided  \( h \) and  \( \mu \) satisfy appropriate bounds. The choice of $R_h$ is flexible and may correspond to projections onto low Fourier modes, local spatial averages, or nodal values. A standard example is the projection onto Fourier modes with wavenumbers   \(|k| \leq 1/h\).
A physically relevant example is given by spatial volume averages over partitions of the domain, as studied in \cite{jones1992determining, jones1993upper, foias1991determining}.

\smallskip

Most rigorous results on continuous data assimilation have been developed in the deterministic setting (see \cite{azouani2014continuous, daley1993atmospheric} and references therein). The stochastic setting, where either the model dynamics or the observations are affected by random perturbations, has received comparatively less attention (see, e.g., \cite{bessaih2015continuous}, \cite{Blomker}); for the numerical results see \cite{Hammoud,CLL21} and the references therein. Nevertheless, stochastic modeling is essential for realistic applications, as measurement noise and unresolved fluctuations are intrinsic features of physical systems.

\smallskip

In the stochastic framework, the system evolution is described by a stochastic differential equation (SDE) of the form
\[
{\rm d}u(t) = F(u(t)) \, {\rm d}t + G(u(t)) \, {\rm d}W(t),
\]
or 
\[
{\rm d}u(t) = F(u(t)) \, {\rm d}t + G \, {\rm d}W(t).
\]
Here \( W(t) \) is a cylindrical Wiener process and \( G \) a suitable operator representing the noise structure, which may depend on the solution or not, thus distinguishing between multiplicative and additive  noise.

The corresponding data assimilation equations take, respectively, the forms
\[
{\rm d}U(t) = \left[ F(U(t)) + \mu R_h(u(t) - U(t)) \right]\,{\rm d}t + G(u(t)) \, {\rm d}W(t),
\]
or 
\[
{\rm d}U(t) = \left[ F(U(t)) + \mu R_h(u(t) - U(t)) \right]\,{\rm d}t + G \, {\rm d}W(t).
\]

In this context, convergence analysis becomes more subtle due to the stochastic forcing. One may study convergence either pathwise (almost surely, for each realization of the noise) or in expectation (in terms of statistical averages). When the noise is additive, the stochastic terms cancel in the evolution of the difference
$u(t)-U(t)$, which allows one to establish pathwise exponential convergence as $t\to+ \infty$, for any  initial data $u(0)$ and $ U(0)$. 
For multiplicative noise, this cancellation no longer occurs, and convergence is typically obtained in expectation, with explicit polynomial or exponential rates depending on the noise growth.

\smallskip

To illustrate these ideas, we focus on the two-dimensional incompressible stochastic Navier-Stokes equations, which provide a canonical testbed: they are analytically tractable yet retain the essential nonlinear features of realistic fluid models. The stochastic formulation offers a natural way to incorporate random perturbations representing unresolved scales, measurement noise, or uncertain forcing. Specifically, we consider a velocity field $u$
evolving on a bounded two-dimensional domain $D$
with no-slip Dirichlet boundary conditions, that satisfies an equation driven by a multiplicative noise:
\begin{align}
\label{eq:multiplicative-noise}
\begin{cases}
{\rm d}u(t) - \nu \Delta u(t) \, {\rm d}t + (u(t) \cdot \nabla) u(t) \, {\rm d}t + \nabla p \, {\rm d}t = f \, {\rm d}t + G(u(t)) \, {\rm d}W(t), 
\\
\text{div } u (t)= 0, 
\\ 
u(0)=u_0,
\end{cases}
\end{align}
or an equation perturbed by an additive noise
\begin{align}
\label{eq:additive-noise}
\begin{cases}
{\rm d}u(t) - \nu \Delta u(t) \, {\rm d}t + (u(t) \cdot \nabla) u(t) \, {\rm d}t + \nabla p \, {\rm d}t = f \, {\rm d}t + G \, {\rm d}W(t), 
\\
\text{div } u(t) = 0, 
\\ 
u(0)=u_0,
\end{cases}
\end{align}
Here the unknowns are the velocity $u$ and the pressure $p$.
The data are the kinematic viscosity  $\nu>0$, the forcing terms ($f$ is the   deterministic one) and the initial velocity.
The stochastic term
$W(t)$ represents random external influences, such as fluctuating wind or temperature variations, and will be modeled as a cylindrical Wiener process. The precise assumptions on $f$,  $G$, and $W$
 are specified in the following sections.

\subsection{Summary of contributions, challenges and main results}
In this work, we first review the deterministic data assimilation results given in \cite{azouani2014continuous} and then present new results for the stochastic setting, distinguishing between additive and multiplicative noise. 
The stochastic setting introduces additional challenges, due to the presence of noise and the inherent uncertainty in the system’s evolution.
We will investigate the convergence of the solution $U$ of the nudging equation towards the  solution $u$ of the true stochastic Navier–Stokes equations, as time goes to infinity.

We begin with the case of multiplicative noise, which represents the most general and technically demanding scenario. In this setting, we establish convergence in expectation of $U$ to $u$, with explicit rates that depend crucially on the growth properties of the noise coefficient $G(u)$.  The proof relies on uniform-in-time a priori estimates for the solution. The convergence behavior of the assimilation algorithm is shown to depend sensitively on the growth conditions imposed on 
$G(u)$. When $G$ is bounded, we obtain exponential convergence in expectation as time goes to infinity. If 
$G$ exhibits sublinear growth, convergence in expectation still holds, though with a polynomial decay rate of arbitrary order 
$p>0$. In the case of linear growth of $G$, convergence persists but only with a polynomial rate, and for 
$p$ restricted to a certain range that depends explicitly on the viscosity $\nu$ and the intensity of the additive component of the noise. In this regime, ensuring convergence requires a structural condition relating the dissipativity of the system to the strength of the stochastic forcing.
\\
These results highlight the greater mathematical complexity of the multiplicative noise framework, where one must balance the destabilizing influence of the stochastic forcing against the stabilizing mechanisms induced by nudging and viscous dissipation. In this case, since the noise term depends explicitly on the state, when taking the difference between $U$ and $u$, the stochastic contributions do not vanish pathwise. Instead, the noise disappears only after taking the expectation, which leads to convergence results in expected value.

We then turn to the case of additive noise, where the situation is considerably simpler. The estimates in expectation follow as a direct corollary of the bounded multiplicative noise case. Moreover, when considering the difference between $U$ and $u$, the additive noise terms cancel out almost surely, allowing for a fully pathwise analysis. This enables us to establish a pathwise exponential convergence result. 

These fundamental differences between the multiplicative and additive noise settings reflect the distinct mathematical challenges posed by the two noise structures, and influence the choice of analytical tools and assumptions required to obtain rigorous data assimilation results.
Taken together, these findings can be viewed as a stochastic generalization of the classical results by \cite{azouani2014continuous}. In particular, the conditions on the nudging parameters $\mu$ and the spatial resolution $h$ now explicitly depend on the intensity of the stochastic forcing. When the noise intensity is set to zero, one recovers, modulo a constant, the deterministic convergence conditions of the original framework \cite{azouani2014continuous}.

We emphasize that the techniques employed in this work to analyze stochastic data assimilation are closely related to those developed for the study of invariant measures, particularly in establishing uniqueness and asymptotic stability. Indeed, both contexts share the fundamental objective of understanding the long-time behavior of stochastic dynamical systems. In this article, we show that the same analytical tools used in the study of invariant measures, especially in the hypoelliptic setting (see e.g. \cite{GHMR17}, \cite{KS},  \cite{FZ23} and \cite{FZ25}), can be effectively adapted to address challenges arising in stochastic data assimilation.

We can summarize our main contributions as follows.
\begin{itemize}
    \item [(i)]In the case of multiplicative  noise,
    \begin{itemize}
        \item 
  if the noise coefficient $G(u)$ is bounded, under suitable assumptions on the nudging parameter $\mu$ and the spatial resolution $h$ (depending on the viscosity $\nu$, as well as on the deterministic and stochastic forcings), we prove that (see Theorem \ref{convergence_mult_case}(i))
    \[
        \mathbb{E}\,\|u(t) - U(t)\|_H^2 \;\longrightarrow\; 0
        \quad \text{exponentially fast as } t \to +\infty.
    \]
    \item In the case of multiplicative  noise, if the noise coefficient $G(u)$ exhibits sublinear growth, then, under appropriate conditions on $\mu$ and $h$ (depending on the viscosity $\nu$, as well as on the deterministic and stochastic forcings), we obtain (see Theorem \ref{convergence_mult_case}(ii))
    \[
        \mathbb{E}\,\|u(t) - U(t)\|_H^2 \;\longrightarrow\; 0
        \quad \text{$p$-polynomially fast for any } p \in (0,+\infty)
        \text{ as } t \to +\infty.
    \]
    \item In the case of multiplicative  noise, if the noise coefficient $G(u)$ exhibits linear growth, and again under suitable conditions on $\mu$ and $h$ (depending on the viscosity $\nu$, as well as on the deterministic and stochastic forcings), we show that (see Theorem \ref{convergence_mult_case}(iii))
    \[
        \mathbb{E}\,\|u(t) - U(t)\|_H^2 \;\longrightarrow\; 0
        \quad \text{$p$-polynomially fast as } t \to +\infty,
    \]
    for exponents $p$ lying in a range that depends explicitly on the viscosity $\nu$ and on the intensity of the stochastic forcing. 
      \end{itemize}
   \item [(ii)] In the case of additive noise, we prove that the solution $U$ of the nudging system converges to the true solution $u$
    both in expectation and pathwise, exponentially fast as $t \to +\infty$, provided that $\mu$ and $h$ satisfy suitable conditions (depending on the viscosity $\nu$, as well as on the deterministic and stochastic forcings). See Theorems \ref{exp-conv-prop} and \ref{pathwise_data_ass}.
\end{itemize}

We specify that exponential convergence means that the limit is dominated by a term of the kind $\tilde Ce^{-Ct}$ for suitable positive constants $C$ and $\tilde C$; similarly, $p$-polynomial decay means that we have a bound in the limit with a 
 term of the kind $\frac C{t^p}$ for suitable positive constant $C$.
 
\subsection{Plan of the paper}
The contents of the paper are structured as follows. Section~\ref{sec:setting} contains all the necessary preliminary material. Section~\ref{sec:det} briefly recalls known results on the data assimilation algorithm in a deterministic framework. Section~\ref{sec:mult} is devoted to investigating the data assimilation algorithm in the stochastic setting in the case of a multiplicative noise, while in Section~\ref{sec:add} the case of an additive noise is investigated. Finally, a collection of useful technical estimates are presented in Appendix~\ref{sec:appA} and Appendix~\ref{sec:appB}. In Appendix \ref{sec:appC} we provide an alternative proof of Theorem \ref{pathwise_data_ass}.

\section{Mathematical setting}
\label{sec:setting}

\subsection{Notations}
Given any Banach space $E$ we denote its topological dual by $E^*$. 
The duality pairing between $E$ and $E^*$ will be indicated by $\langle\cdot,\cdot\rangle _{E^*,E}$. If no confusion arises, we may drop the subscripts.
For any real Hilbert space $H$, we denote by $\|\cdot\|_H$ and $(\cdot, \cdot)_H$ 
the norm and the scalar product, respectively. 
Given any two Banach spaces 
$E$ and $F$, we use the symbol $\mathcal{L}(E,F)$ for the space of linear bounded operators from $E$ to $F$. 
If $H$ and $K$ are separable Hilbert spaces, we employ the symbol
$L_{HS}(H, K)$ for the space of Hilbert--Schmidt operators from $H$ to $K$, endowed with its canonical norm $\|\cdot\|_{L_{HS}(H, K)}$. 
For a fixed $T>0$ we denote by 
$\mathcal{C} ([0,T];E)$ the space of
functions from $[0,T]$ to $E$ which are 
strongly continuous.
Moreover, for every $p\in[1,+\infty)$
we use the classical symbol $L^p(0,T; E)$ for the space
of strongly measurable
$p$-Bochner integrable functions from $(0,T)$ to $E$. Moreover by $L^p_{loc}(0,+\infty; E)$ we denote the space of locally $p$-Bochner integrable functions. Similarly for $p=\infty$.

 As for the probabilistic framework, let $(\Omega,\cF,\mathbb{F}:=(\cF_t)_{t\ge 0},\mathbb{P})$ be a filtered probability space satisfying the usual conditions (namely it is saturated and right-continuous).

 Throughout the whole work,  we reserve the symbol $C$ (eventually indexed) for constants depending on some or all the structural parameters of the problem. If the dependencies of such constants are relevant, they will be explicitly pointed out. The values of these constants may change within the same argument without relabelling.

\subsection{Functional spaces and operators}
Consider a bounded domain $D\subset\mathbb{R}^2$ with smooth boundary $\partial D$ and Lebesgue measure denoted by $|D|$. The symbol $W^{s,p}(D)$, where $s\in\mathbb{R}$ and $p\in[1,+\infty]$, denotes the usual real Sobolev space of order $(s,p)$ and we denote by $\|\cdot\|_{W^{s,p}(D)}$ its canonical norm. In the Hilbert case $p = 2$, we define $H^s(D):=W^{s,2}(D)$, $s\in\mathbb{R}$,
	endowed with its canonical norm $\|\cdot\|_{H^s(D)}$.

    In order to handle the Navier--Stokes velocity field, we define the following solenoidal vector-valued spaces
	\begin{align*}
		H := \overline{\{ v \in  [{\mathcal{C}}^\infty_0(D)]^2: \text{div}\ v = 0 \text{ in } D\}}^{[{L}^2(D)]^2}, \quad
		V := \overline{\{ v \in [{\mathcal{C}}^\infty_0(D)]^2 : \text{div}\ v = 0 \text{ in } D\}}^{[{H}^1(D)]^2}.
	\end{align*}
    The space $H$ is endowed with the Hilbert structure inherited by $[L^2(D)]^2$, and we denote its norm and scalar product by $\|\cdot\|_H$ and $(\cdot,\cdot)_H$, respectively.
	By means of the Poincaré inequality, on the space $V$ we can use the norm $\|v\|_{V}:=\|\nabla v\|_{[L^2(D)]^2}$, for all $v\in V$, and its respective
	scalar product $(\cdot,\cdot)_{V}$. By identifying the Hilbert space $H$ with its dual we have the variational  structure
	\[
	V \subset H \subset  V^*,
	\]
	 with dense and compact embeddings. 
     
\textbf{The Stokes operator $A$.}     We denote by $\Pi$ the Leray-Helmholtz projector of $[L^2(D)]^2$ onto $H$.
The Stokes operator $A=-\Pi\Delta$ is a linear operator in $H$ with domain $D(A)=[H^2(D)]^2\cap V$. 
It can be extended (and we use the same notation) as a linear operator $A:V\to V^*$ by
\[
	\langle Av, w \rangle_{V^*,V}:=
	(\nabla v, \nabla w)_{ H}\,, \qquad \forall \: v,w\in V.
	\]
Recalling the spectral properties of the operator $A$, it is possible to define the family of power operators $A^s$ for any $s \in \mathbb{R}$. In particular, using $H$ as pivot space, we highlight the general structure
		\[
		D(A^s) \subset D(A^t) \subset H \equiv D(A^0) \subset  D(A^{-t}) \subset D(A^{-s}),
		\]
		for any $s > t > 0$, with dense and compact embeddings.

We denote by $\{\lambda_n\}_{n \in \mathbb{N}}$ the eigenvalues of  the Stokes operator 
$ A$ and by $\{e_n\}_{n \in \mathbb{N}}$ the corresponding eigenvectors, which form a complete orthonormal system in $H$. Moreover, $0<\lambda_1\le \lambda_2 \le \cdots,$ and 
\[
\lim_{n \rightarrow+ \infty}\lambda_n=+ \infty.
\]
We have  the Poincar\'e inequality 
\begin{align}
    \label{Poincarè}
\lambda_1 \left\| u \right\|_H^2 \leq \left\| u \right\|_V^2,  \quad u \in V.
\end{align}

\textbf{The bilinear form $B$.} We define the  trilinear form 
	$b$ on $V \times V \times V$ as
	\[
	b(u,  v,  w):=\int_{D}( u\cdot\nabla) v\cdot w
	=\sum_{i,j=1}^2\int_D u_i\frac{\partial v_i}{\partial x_j}w_j\,,
	\qquad \forall \:  u,  v,  w \in V\,,
	\]
	and the associated bilinear form $B:V\times V\to  V^*$ as
	\[
	\langle B(u, v), w\rangle_{ V^*, V}:=b(u,  v,  w)\,,
	\qquad\forall \: u, v, w \in V\,.
	\]
	Let us recall that $b(u, v,  w)=-b(u,  w, v)$ for all
	$ u,  v,  w \in V$, from which it follows in particular
	that $b(u, v, v)=0$ for all $ u,  v\in V$. 
  As a consequence, the map  $B$ satisfies the following antisymmetry relations:
  \begin{equation} 
  \label{B}
  \langle B(u,v), w\rangle = - \langle B(u,w), v\rangle ,  \quad \text{and } \quad \langle B(u,v),  v\rangle = 0, \qquad \mbox {for all} \ 
 u,v, w\in V.
  \end{equation}

\subsection{Assumptions on the stochastic term}
Eventually, let us deal with the stochastic term. 
On the normal filtered probability space $(\Omega,\cF,(\cF_t)_{t\ge 0},\mathbb{P})$, we consider a cylindrical Wiener process $W$ with values in some separable Hilbert space $U$. 
For convenience, we fix once and for all a complete orthonormal system
	$\{u_k\}_{k\in\mathbb{N}}$ on $U$.
    We recall that, as a cylindrical process on $U$, $W$ admits the following formal representation
	\begin{equation} \label{eq:representation}
		W = \sum_{k\in\mathbb{N}} \beta_k u_k,
	\end{equation}
	where $\{\beta_k\}_{k \in \mathbb{N}}$ is a family of real and independent Brownian motions. It is well known that \eqref{eq:representation} does not converge in $U$. Nonetheless, there always exists a larger Hilbert space $U_0$, such that
	$U \subset U_0$ with Hilbert--Schmidt embedding $\iota$, enabling the identification of $W$ as a $Q^0$-Wiener process on $U_0$, with $Q^0=\iota \circ \iota^*$ being trace-class on $U_0$ (see \cite[\S 2.5.1]{LiuRo}). 
In the following, we may implicitly assume this extension by simply saying that $W$ is a cylindrical process on $U$. Similarly, one deals with the  stochastic integral $
	\int_0^t G(s)\,{\rm d} W(s)
$ (see \cite[\S 2.5.2]{LiuRo}). 

\subsubsection{The multiplicative noise case}
When dealing with a multiplicative type noise, that is a noise that depends on the solution, we impose the following conditions on the covariance operator, i.e. on the operator $G$ appearing in equation \eqref{eq:multiplicative-noise}.
\begin{Assumption}
\label{assumption-stochastic1}
$G:{H} \to L_{HS}(U,H)$ is a Lipschitz continuous operator,  i.e.
 \begin{equation} \label{Lipschitz_G}
  \exists \ L >0 : \quad  \|G(u)-G(v)\|^2_{L_{HS}(U,H)} \le L\|u-v\|^2_H \qquad \forall u, v \in H.
 \end{equation}
\end{Assumption}

\begin{Assumption}
\label{assumption-stochastic2}
One of the following conditions is in force.
 $\;$\\[-3mm]
 \begin{itemize}
\item[\textbf{(B)}] \label{item:1} [Bounded noise]
 There exists a positive constant $K_B$ such that
 \begin{equation*}\label{bounded}
  \|G(u)\|^2_{L_{HS}(U,H)}\le K_B,\quad \forall u \in H.
 \end{equation*}
\item[\textbf{(S)}] \label{item:2}  [Sublinear noise]
  There exist non negative constants $K_S, \widetilde{K_S}$ and $\gamma \in (0,1)$ such that
 \begin{equation*}
  \|G(u)\|^2_{L_{HS}(U,H)}\le K_S+ \widetilde{K_S}\|u\|_H^{2\gamma}, \quad \forall u \in H.
 \end{equation*}
\item[\textbf{(L)}]  [Linear noise]
\label{item:3} 
  There exist non negative constants $K_L, \widetilde{K_L}$ such that
 \begin{equation*}
  \|G(u)\|^2_{L_{HS}(U,H)}\le K_L+ \widetilde{K_L}\|u\|^2_H, \quad \forall u \in H.
 \end{equation*}
\end{itemize}
\end{Assumption}

In the sequel, when we say that Assumption \ref{assumption-stochastic2} holds we mean that 
one of the three assumptions (B),(S) or (L) holds.
Concrete examples of operators $G$ satisfying the above conditions can be found in \cite[Example 2.4]{FZ23}.
Notice that
$(\bf L) \Longrightarrow(\bf S) \Longrightarrow(\bf B) $; so we  could consider only part $(\bf L)$ in Assumption 
\ref{assumption-stochastic2}. However the  convergence results are different in the three cases and with assumption
$(\bf L)$ we obtain the most restrictive result.

\subsubsection{The additive noise case}
When dealing with an additive type noise, that is a noise that does not depend on the solution, we
 assume that $G$ is a Hilbert-Schmidt operator from $U$ into $H$. This might be seen as a particular case of the multiplicative bounded  noise (taking $L=0$ and $K_B= \|G(u)\|^2_{L_{HS}(U,H)}$).

\subsection{Interpolant operators}
\label{sect:interpolant-operators}
In the absence of the knowledge of the initial  velocity $u_0$, the data assimilation method for
the two-dimensional incompressible Navier-Stokes equations allows
the incorporation of the following linear operator $R_h: [H^1(D)]^2 \longrightarrow [L^2(D)]^2$, 
which interpolates its input function at a length scale $h>0$
and satisfies the approximating identity property
 
 \begin{equation}
\label{iden-data-approx}
\left\| \phi -  R_h(\phi) \right\|_{[L^2(D)]}^2 \leq c_0 h^2 \left\| \phi \right\|_{[H^1(D)]^2}^2,
 \end{equation}
 for every $\phi \in [H^1(D)]^2.$ 
 

A physically relevant example of an interpolant which satisfies \eqref{iden-data-approx} is given by finite volume elements  (see \cite{azouani2014continuous} and the therein references). 
Specifically, let $h > 0$ be given, and let 
\[
D = \bigcup_{j=1}^{N_h} Q_j,
\]
where the $Q_j$ are disjoint subsets such that $\operatorname{diam}(Q_j) \le h$ for 
$j = 1, 2, \ldots, N_h$. Then we set
\[
R_h(\varphi)(x) = \sum_{j=1}^{N_h} \overline{\varphi}_j\,\pmb{1}_{Q_j}(x),
\]
where 

\[
 \overline{\varphi}_j=\frac{1}{|Q_j|}\int_{Q_j} \varphi (x) dx
 \]

The orthogonal projection onto the Fourier modes, with wave numbers k
such that $|k| \leq \frac{1}{h},$ is an another example of an interpolant operator which satisfies  the approximation property 
\eqref{iden-data-approx}.

\section{Data assimilation in the deterministic framework}
\label{sec:det}

Rigorous analysis of the nudging algorithm for partial differential equations in fluid mechanics began with
the work of Azouani, Olson and Titi \cite{azouani2014continuous}, where the authors estimate  threshold values for the relaxation parameter $\mu$ and data resolution $h$ for the two dimensional incompressible Navier-Stokes equations. 

More in details, in \cite[Section 3]{azouani2014continuous} the authors consider the  Navier-Stokes equations  
\begin{align}
\label{deterministic-NS}
\begin{cases}
  \partial_t u   - \nu  \Delta u +  \left(u \cdot \nabla \right) u   + \nabla p \,   = f    
  \\
 {\rm{div}} \; u = 0, 
 \\
 u(0)=u_0,
 \end{cases}
\end{align}
on an open, bounded and connected set $D$ in $\mathbb{R}^2$ with $C^2$-boundary, endowed with no-slip Dirichlet boundary conditions.
They introduce the nudging equation as 
\begin{align}
\label{deterministic-NS_nud}
\begin{cases}
  \partial_t U   - \nu  \Delta U +  \left(U \cdot \nabla \right) U   + \nabla p \,   = f   - \mu R_{h}(U-u)  
  \\
 {\rm{div}} \; U= 0, 
 \\
 U(0)=U_0,
 \end{cases}
\end{align}
and formulate their  main result  as follows.
\begin{theorem}
\label{AOT_thm}
Let $f \in L^\infty(0,+ \infty;H)$ and let $R_h$ satisfy assumption \eqref{iden-data-approx}.
Let $u$ be the solution to the equation \eqref{deterministic-NS} with initial condition $u_0 \in H$ and  $U$ the solution to the equation \eqref{deterministic-NS_nud} with initial condition $U_0\in H$. Then, if $\mu$ and $h$ satisfy
\begin{equation}
\label{cond_mu_det}
c\mathcal{G}^2\nu \lambda_1 \le  \mu \leq \frac{  \nu }{c_0 h^2 },
\end{equation}
where $c$ is a suitable positive constant, $c_0$ is the constant appearing in \eqref{iden-data-approx} and $\mathcal{G}$ denotes the Grashof number
\[
\mathcal{G}:= \frac{1}{\nu^2 \lambda_1} \limsup_{t \rightarrow +\infty} \left\| f(t) \right\|_H,
\]
then
$\|u(t)-U(t)\|_H \rightarrow 0$ exponentially fast as $t \rightarrow + \infty$.
\end{theorem}

\section{Data assimilation in the stochastic framework: the multiplicative noise case}
\label{sec:mult}

We start with  the study of the 2D Navier-Stokes equations driven by a multiplicative noise. Projecting equation \eqref{eq:multiplicative-noise} onto $H$, we get rid of the pressure term and we obtain the following abstract
formulation for the equations \eqref{eq:multiplicative-noise}
\begin{equation}
\label{NS_abs_mult}
\begin{cases}
{\rm d}u(t) + \left[\nu Au(t)+B(u(t),u(t))\right]\,{\rm d}t=f \, {\rm d}t+G(u(t))\,{\rm d}W(t),\qquad\qquad t>0
\\
u(0)=u_0\in H
\end{cases}
\end{equation}
We assume $\nu>0$   
and $f\in V^*$ independent of time.

 \begin{theorem}
 \label{NS_mult-wp}
Let Assumptions \ref{assumption-stochastic1} and \ref{assumption-stochastic2} be in force, 
let $f \in V^*$ and $u_0 \in H$. 
Then  there exists a unique strong solution $u$ to equation \eqref{NS_abs_mult} with $\PP$-a.s. paths in 
\begin{equation*}
\mathcal{C}([0,+\infty), H) \cap L_{loc}^2(0,+\infty;V),
\end{equation*}
that $\mathbb{P}$-a.s. satisfies
  \begin{multline*}			
  \left(u(t), \phi\right)_H +  \int_0^t \left( A^{1/2} u(s), A^{1/2} \phi\right)_H\, {\rm d}s - \int_0^t \big\langle B(u(s),\phi),  u(s) \big\rangle \,{\rm d}s
  \\
=  \left( u_0, \phi\right)_H + \langle f, \phi\rangle t+\int_0^t\langle G(u(s))\,{\rm d}W(s), \phi \rangle,
  \end{multline*}
for every $t \ge 0$ and every $\phi \in V$. 
\end{theorem}
\begin{proof}
    See \cite[Theorem 3.6]{FZ23}.
\end{proof}

The solution given in the previous theorem is a strong solution in the probabilistic sense, and it is an adapted process. 
Since we are working in the intersection of analysis and probability, the terminology concerning the notion of solutions can cause some confusion. 
When we talk of a strong solution we understand it in a probabilistic sense, that is we mean that the underlying probability space is given in advance. 
On the other hand, the term weak solution is used as a synonym of martingale solution: we mean that the stochastic basis is constructed as part of the solution. In both cases solutions are weak in the PDE sense.
\\
Mean square estimates for the $H$ and $V$-norm of the solution $u$ are given in Appendix \ref{sec:appA}.

From now on we assume $u_0\in H$ and $ f \in V^*$. We will not repeat these assumptions in the next statements. Only the assumptions on the covariance of the noise will be specified in our next results.

We introduce the data assimilation equation as follows. We consider the solution $u$ of the Navier-Stokes equation \eqref{NS_abs_mult} and define the process $U$ solving the following equation, which we call the data assimilation equation
\begin{equation}
\begin{cases}
\label{data-assi-NS-equation-abs_mult}
{\rm d} U(t)+ \left( \nu A U(t) + B(U(t),U(t)) \right) {\rm d}t  =  G(U(t)){\rm d} W(t) + f {\rm d}t - \mu \Pi R_{h}\left(U(t)-u(t)\right)   {\rm d}t    ,
\\
U(0)=U_0\in H,
\end{cases}
\end{equation}
where $u$ is the solution of the Navier-Stokes equation \eqref{NS_abs_mult} and $R_h$, defined in Section \ref{sect:interpolant-operators}, fulfills \eqref{iden-data-approx}.

Notice that when $R_h$ fulfills condition \eqref{iden-data-approx}, equation \eqref{data-assi-NS-equation-abs_mult} admits a unique strong solution with the same regularity as the solution to equation \eqref{NS_abs_mult}. This can be easily proved by noticing that the additional term $\mu \Pi R_{h}(U(t)-u(t))$ does not crucially impact the well-posedness estimates.

\subsection{Convergence in expectation}
Dealing with a multiplicative noise, the convergence analysis is conducted in expectation, meaning that convergence results hold in expected value. Provided the observational resolution $h$ and the nudging strength $\mu$ satisfy suitable conditions, we show that $\mathbb{E}\left[ \|u(t)-U(t)\|^2_H\right]$
vanishes as $t\to+\infty$. 
The convergence rate depends on the growth of $G$ as specified in Assumption \ref{assumption-stochastic2}.
In the context of data assimilation the estimate $\mathbb{E}\left[ \|u(t)-U(t)\|^2_H\right]$ is often referred to as Foias-Prodi estimate in expected value. 
This estimate shows that the control term $\mu \Pi R_h(U(t)-u(t))$, when the parameters $\mu$ and $h$ are chosen in a proper way, allows to synchronize (in expectation) the solution $u$ of the Navier-Stokes equation \eqref{NS_abs_mult} and 
 the solution $U$ of the data assimilation equation \eqref{data-assi-NS-equation-abs_mult} in the limit as $t \rightarrow +\infty$. 

The main result of this Section reads as follows.
\begin{theorem}
\label{convergence_mult_case}
Let Assumption \ref{assumption-stochastic1} be  satified and assume that $R_h$ satisfies assumption \eqref{iden-data-approx}. 
Let $u$  and $U$ be, respectively, the solution to the Navier-Stokes  equation \eqref{NS_abs_mult} and 
 to the data assimilation  equation \eqref{data-assi-NS-equation-abs_mult}. 
\begin{enumerate}[label=$(\roman{*})$] 
\item  Assume assumption \ref{assumption-stochastic2}(B) holds.
If  $h$  and $\mu$ satisfy the condition 
\begin{equation}
\label{cond_mu_bounded}
\left(L+ \frac{C_B}{\nu^2}\right)< \mu \le \frac{  \nu }{c_0 h^2 },
\end{equation}
where 
\begin{equation}
\label{C_B}
C_B:=2\left(K_B+ \frac{1}{\nu}\|f\|^2_{V^*}\right), 
\end{equation} 
then
$\mathbb{E}\left[\|u(t)-U(t)\|_H^2\right] \rightarrow 0$ exponentially fast as $t \rightarrow + \infty$.
\item   Assume asssumption \ref{assumption-stochastic2}(S) holds.
If $h$  and $\mu$ satisfy the condition 
\begin{equation}
\label{cond_mu_sublinear}
\left(L+ \frac{C_S}{\nu^2}\right)< \mu \leq \frac{  \nu }{c_0 h^2 }, 
\end{equation}
where
\begin{equation}
\label{C_S}
C_S:=2\left(1+ K_S + 
(1-\gamma)\left(\frac{2\gamma}{\lambda_1 \nu}\right)^{\frac{\gamma}{1-\gamma}} \widetilde{K_S}^{\frac{1}{1-\gamma}}
 +
\frac{2}{\nu}\|f\|^2_{V^*}\right) ,
\end{equation} 
then
$\mathbb{E}\left[\|u(t)-U(t)\|_H^2\right] \rightarrow 0$ $p$-polynomially fast as $t \rightarrow + \infty$, for any power $p \in (0, +\infty)$.
\item  Assume assumption \ref{assumption-stochastic2}(L) holds. If  
\begin{equation}
\nu >\frac{6\widetilde{K}_L}{\lambda_1}
\end{equation}
  and 
 $h$  and $\mu$ satisfy the condition 
\begin{equation}
\label{cond_mu_linear}
\left(L+ \frac{C_L}{\nu\left(\nu-\frac{\widetilde{K}_L}{\lambda_1}\right)}\right)<  \mu \le \frac{  \nu }{c_0 h^2 }, 
\end{equation}
where 
\begin{equation}
\label{C_L}
 C_L:=2\left(1+  K_L + \frac{1}{\nu }\|f\|^2_{V^*}\right),
\end{equation} 
then
$\mathbb{E}\left[\|u(t)-U(t)\|_H^2\right] \rightarrow 0$ $p$-polynomially fast as $t \rightarrow + \infty$, for any power $p \in \left(0, \frac{\nu \lambda_1}{4\widetilde{K}_L}-\frac 12 \right)$.
\end{enumerate}
\end{theorem}

To prove the above result, we need some estimates, both for the difference $U-u$ 
and for the solution $u$ of the Navier-Stokes equation \eqref{NS_abs_mult}.
Indeed, as we will see in the next proof, the difference process 
$r=U-u$ fulfills a stochastic equation with noise (unless the noise is additive, see the next section), and the basic estimates on the $H$-norm of $r$ depend on the term $\nu \int_0^t \|u(s)\|_V^2 {\rm d}s$, see \eqref{est1mul}.  
We emphasize once more that, due to the presence of a multiplicative noise, we obtain estimates in the mean square sense. 
Only in the next section, when we consider an additive noise, also pathwise estimates can be obtained.

We begin with a result involving a generic stopping time; later on, we will choose it in a suitable way. 
\begin{lemma}
   \label{data-estimate-2} Let $G$ satisfy Assumptions \ref{assumption-stochastic1} and \ref{assumption-stochastic2},  and $R_h$ satisfy assumption \eqref{iden-data-approx}.
Let $u$ be the solution to the Navier-Stokes equation \eqref{NS_abs_mult}  and $U$ be the solution to the data assimilation equation \eqref{data-assi-NS-equation-abs_mult}. 
\\
If
\begin{equation*}
\label{condi-1-mu2}
0<\mu \le \frac{  \nu }{c_0 h^2 },
\end{equation*}
where $c_0$ is the constant that appears in estimate \eqref{iden-data-approx}, 
then the following estimate holds for any $\delta\ge 0$, any stopping time $\tau$ and any $t \ge 0$:
\begin{align}
\label{error-u-U_mult}
\mathbb{E}\left[e^{\left(\frac{\mu}{1+\delta}-L\right)(t\wedge \tau)-\frac{1}{\nu}\int_0^{t \wedge \tau}\|u(s)\|^2_V\, {\rm d}s}\left\| u(t\wedge \tau)-U(t\wedge \tau) 
\right\|_{H}^2\right] 
&\leq \left\|u_0 - U_0 \right\|_{H}^2.
\end{align}     
\end{lemma}
\begin{proof}
    Given $u$ and $U$ satisfying the equations \eqref{NS_abs_mult} and \eqref{data-assi-NS-equation-abs_mult}, respectively,  we obtain the evolution of the difference $r(t)=U(t)-u(t)$ 
\begin{equation*}
\begin{cases}
{\rm d} r(t) + \left[\nu A r(t) + \left( B(U(t),U(t)) - B(u(t),u(t)) \right) \right]\, {\rm d}t = - \mu \Pi R_h(r(t)) + (G(U(t))-G(u(t)))\, {\rm d}W(t)  \\ 
r(0)=U_0 - u_0 \in H.
\end{cases}
\end{equation*}
We apply the It\^o formula to the functional $\|\cdot\|^2_H$. 
Using the fact that $B(U,U) - B(u,u)= B(U,r)+B(r,u)$ and bearing in mind \eqref{B}, we obtain for every $t \ge 0$, $\mathbb{P}$-a.s.,
\begin{align*}
 {\rm d}\|r(t)\|^2_H &+ 2\nu \|r(t)\|^2_V\,{\rm d}t 
 \\
 &= \left[ -2 \langle r(t),B(r(t),u(t)) \rangle   - 2 \mu \langle \Pi R_h(r(t)),r(t) \rangle +  \|G(U(t))-G(u(t))\|^2_{L_{HS}(U,H)} \right]\, {\rm d}t
 \\
 &+ \langle r(t), [G(U(t))-G(u(t))]\, {\rm d}W(t)   \rangle.
\end{align*}
Thanks to the Gagliardo-Nirenberg, Cauchy-Schwartz and Young inequalities we estimate
\begin{align*}
 -2& \langle r(t),B(r(t),u(t)) \rangle   - 2 \mu \langle \Pi R_h(r(t)),r(t) \rangle
 \\
 &\le 2  \left| \langle r(t), r(t) \cdot \nabla u(t) \rangle \right| - 2 \mu \left\| r(t) \right\|_{H}^2 - 2 \mu \langle \Pi R_h(r(t))-r(t),r(t)\rangle
\\ 
& \le  2 \left\| r(t) \right\|_{[L^4(D)]^2}^2 \left\| \nabla u(t) \right\|_{H}  - 2 \mu \left\| r(t) \right\|_{H}^2 + 2 \mu \left\| r(t) \right\|_{H} \left\|\Pi R_{h}(r(t)) - r(t) \right\|_{H} \\ 
& \le  2 \left\| r(t) \right\|_{H} \left\| \nabla r(t) \right\|_{H}   \left\| \nabla u(t) \right\|_{H}    
  - 2 \mu \left\| r(t) \right\|_{H}^2+2 \mu \left\| r(t) \right\|_{H} \left\| \Pi R_{h}(r(t)) - r(t) \right\|_{H} \\ 
& \le 2 \left( \frac{\nu}{2} \left\| \nabla  r(t) \right\|_{H}^2  + \frac{1}{2 \nu} \left\| r(t) \right\|_{H}^2 \left\| \nabla u(t) \right\|_{H}^2    \right) - 2 \mu \left\| r(t) \right\|_{H}^2
 + 2 \mu \left( \frac{1}{2}  \left\| r(t) \right\|_{H}^2  +
\frac{1}{ 2 } \left\|\Pi R_{h}(r(t)) - r(t) \right\|_{H}^2  \right)  \\ 
 & \le \nu \left\| \nabla  r(t) \right\|_{H}^2  + \frac{1}{\nu} \left\| r(t) \right\|_{H}^2 \left\| \nabla u(t) \right\|_{H}^2 - \mu  \left\| r(t) \right\|_{H}^2 + \mu \left\|\Pi R_h(r(t)) -r(t) \right\|_{H}^2.
 \end{align*}
Thus we obtain the estimate 
\begin{align*}
    {\rm d}\|r(t)\|^2_H &+ \left[\nu \|r(t)\|^2_V + \mu \|r(t)\|^2_H\right] \,{\rm d}t
    \le \langle r(t), [G(U(t))-G(u(t))]\, {\rm d}W(t)   \rangle
    \\
    &+\left[\frac{1}{\nu}\|u(t)\|_V^2 \|r(t)\|^2_H + \mu \|\Pi R_h(r(t))-r(t))\|^2_H + \|G(U(t))-G(u(t))\|^2_{L_{HS}(U,H)} \right]\, {\rm d}t.
\end{align*}
Exploiting the fact that $R_h$ satisfies \eqref{iden-data-approx} we have 
\[
\left\| \Pi R_h(r(t)) - r(t) \right\|_{H}^2=\left\| \Pi (R_h(r(t)) - r(t))\right\|_{H}^2 
\le \left\| R_h(r(t)) - r(t) \right\|_{[L^2(D)]^2}^2 
\le c_0 h^2  \left\| r(t) \right\|_{V}^2.
\]
Therefore, exploiting Assumption \ref{assumption-stochastic1} we obtain the estimate  
\begin{equation*}
{\rm d} \left\| r(t) \right\|_{H}^2 + \left(\nu - c_0\mu h^2\right) \left\| r(t) \right\|_{V}^2  + \left(\mu-L-\frac{1}{\nu}\|u(t)\|^2_V\right)  \left\| r(t) \right\|_{H}^2 dt 
 \le  \langle r(t), [G(U(t))-G(u(t))]\, {\rm d}W(t)  \rangle.
\end{equation*}
If we assume $\mu \leq \frac{  \nu }{c_0 h^2 }$, we thus obtain 
\begin{equation}
\label{est1mul}
{\rm d} \left\| r(t) \right\|_{H}^2  + \left(\mu-L-\frac{1}{\nu}\|u(t)\|^2_V\right)  \left\| r(t) \right\|_{H}^2 dt 
 \le  \langle r(t), [G(U(t))-G(u(t))]\, {\rm d}W(t)  \rangle.
\end{equation}
Let us set 
\[
M(t):=\int_0^t  \langle r(s), [G(U(s))-G(u(s))]\, {\rm d}W(s)  \rangle, \quad t \ge 0.
\]
Fix $\delta\ge0$ and define
\[ 
\Gamma (t):= \left(\frac{\mu}{1 + \delta}-L\right)t-\frac{1}{\nu}\int_0^t\|u(s)\|^2_V\,{\rm d}s, \quad t \ge 0.
\]
We rewrite \eqref{est1mul} as 
\begin{equation*}
{\rm d}\|r(t)\|^2_H + \left(\frac{\delta\mu}{1+\delta} \|r(t)\|^2_H+ \Gamma^\prime(t)\|r(t)\|^2_H\right)  {\rm d}t
 \le {\rm d}M(t).
\end{equation*}
Multiplying  both members of the above expression by $e^{\Gamma(t)}$ and noticing that 
${\rm d}(e^{\Gamma(t)} \|r(t)\|^2_H)
=
e^{\Gamma(t)} {\rm d}\|r(t)\|^2_H+ \Gamma^\prime(t)e^{\Gamma(t)} \|r(t)\|^2_H {\rm d}t$, 
we get
\begin{equation}\label{Ito-per-prodotto}
{\rm d}\left(e^{\Gamma(t)}\|r(t)\|^2_H\right)+ \frac{\delta\mu}{1+\delta}e^{\Gamma(t)}\|r(t)\|^2_H {\rm d}t
\le e^{\Gamma(t)}{\rm d}M(t)
\end{equation}
i.e.
\[
e^{\Gamma(t)}\|r(t)\|^2_H +\frac {\delta\mu}{1+\delta} \int_0^t e^{\Gamma(s)}\|r(s)\|^2_H {\rm d}s
\le
\|r_0\|^2_H +\int_0^t e^{\Gamma(s)}{\rm d}M(s).
\]
In particular
\[
e^{\Gamma(t)}\|r(t)\|^2_H 
\le
\|r_0\|^2_H +\int_0^t e^{\Gamma(s)}{\rm d}M(s).
\]
Using usual strategies for local martingales, considering a stopping time $\tau$ and taking the expected value, we infer 
\begin{equation}
\label{star1}
\mathbb E \left[e^{\Gamma(t\wedge \tau)} \|r(t\wedge \tau)\|_H^2\right]
\le 
\|r_0\|_H^2,
\end{equation}
that is \eqref{error-u-U_mult}.
\end{proof}

In order to control the integrating factor that appears in \eqref{error-u-U_mult}, we will make a suitable choice of the stopping time. For $R, \beta, \delta, \mu>0$, let
\begin{equation}
\label{tau_R_beta}
\tau_{R, \beta, \delta\, \mu} := 
\inf \left\{r \ge 0: \frac{1}{\nu} \int_0^r \|u(s)\|^2_V\, {\rm d}s +\left(L- \frac{\mu}{(1+\delta)^2} \right)r - \beta \ge R\right\}
\end{equation} 
and $\tau_{R, \beta, \delta, \mu}=+\infty$ if the  set is empty, i.e. if 
\begin{equation}\label{tau=infty}
\frac{1}{\nu} \int_0^t \|u(s)\|^2_V\, {\rm d}s +\left(L- \frac{\mu}{(1+\delta)^2} \right)t - \beta < R,
\qquad \forall \ t\ge 0.
\end{equation}
The parameter $\beta$ will be useful to track the dependence on the initial velocities 
and on the forcing terms in subsequent estimates of $\tau_{R, \beta, \delta, \mu}$; see Proposition \ref{tau_exp} below.

From the definition \eqref{tau_R_beta} of $\tau_{R,\beta,\delta, \mu}$  we immediately get the following corollary of Lemma \ref{data-estimate-2}.
\begin{coro}
\label{cor_tau_R_beta}
Under the same conditions as in  Lemma \ref{data-estimate-2}, for any $u_0, U_0 \in H$ and any $R, \beta, \delta > 0$ we have
\begin{equation*}
\mathbb{E} \left[{\pmb 1}_{(\tau_{R, \beta,\delta,\mu}=+\infty)}\|u(t)-U(t)\|^2_H\right] 
\le 
e^{R+ \beta -\frac{\delta\mu}{(1+\delta)^2} t}\|u_0-U_0\|^2_H.
\end{equation*}
\end{coro}
\begin{proof}
From \eqref{tau=infty} we know that 
$ \frac{\delta\mu}{(1+\delta)^2} t-\beta-R\le \Gamma(t)$ for any $t\ge 0$, when $\tau_{R, \beta, \delta, \mu}=+\infty$. Hence, from \eqref{star1} we get
\[
\mathbb E\left( {\pmb 1}_{(\tau_{R, \beta, \delta, \mu}=+\infty)} e^{\frac{\delta\mu}{(1+\delta)^2} t-\beta-R}\|u(t)-U(t)\|^2_H\right)
\le
\mathbb E\left( {\pmb 1}_{(\tau_{R, \beta,\delta, \mu}=+\infty)} e^{\Gamma(t)}\|u(t)-U(t)\|^2_H\right)
\le \|u_0-U_0\|^2_H
\]
and the thesis follows.
\end{proof}
In the following result,  for a suitable choice of the parameters $\beta, \delta$ and $\mu$, we estimate the probability $\mathbb{P}(\tau_{R, \beta,\delta, \mu}<+\infty)$ in terms of the parameter $R$. We obtain different decays according to the different assumptions on the operator $G$.

\begin{prop}
\label{tau_exp}
Let us assume assumption \ref{assumption-stochastic1} holds.
Given arbitrary $R>0$ and $\delta>0$, consider the stopping time $\tau_{R, \beta, \delta, \mu}$ defined in \eqref{tau_R_beta}, 
where $u$ is the solution of the Navier-Stokes equation \eqref{NS_abs_mult} starting from $u_0$.
\begin{itemize}
\item [(i)]  Assume assumption \ref{assumption-stochastic2}(B) holds. Let $C_B$ be the constant defined in \eqref{C_B}.
If 
\begin{equation}
\label{cond_beta_i}
\beta\ge\frac{2}{\nu^2} \|u_0\|^2_H
\end{equation}
and
\begin{equation}
\label{cond_mu_i}
\mu \ge (1+\delta)^2\left(L + \frac{C_B}{\nu^2}\right),
\end{equation}
then 
\begin{equation*}
\label{tau_est_i}
\mathbb{P}(\tau_{R, \beta, \delta, \mu}< +\infty) 
\le 
e^{-\frac{\nu^3\lambda_1}{16K_B}R}.
\end{equation*}
\item [(ii)]  Assume assumption \ref{assumption-stochastic2}(S) holds. Let $C_S$ be the constant defined in \eqref{C_S}.
If 
\begin{equation}
\label{cond_beta_ii}
\beta\ge\frac{1}{\nu^2} \left(\|u_0\|^2_H+ C_S\right)
\end{equation}
and 
\begin{equation}
\label{cond_mu_ii}
\mu \ge  (1+\delta)^2 \left(L+\frac{C_S}{\nu^2} \right),
\end{equation}
then 
\begin{equation*}
\label{tau_est_iii}
\mathbb{P}(\tau_{R, \beta, \delta, \mu}< +\infty) 
\le 
 \frac{C(1+ \|u_0\|_H^{4(p+1)})}{R^p},
 \end{equation*}
 for any $p>0$,  where $C$ is a positive constant depending on $p, \nu, K_S, \tilde{K}_S, \gamma, \|f\|_{V^*}$
 and  independent of $R, \beta$ and $u_0$.
 \item [(iii)] Let $\nu > \frac{2\widetilde K_L}{\lambda_1}$ and assumption \ref{assumption-stochastic2}(L)  holds. Let $C_L$ be the constant defined in \eqref{C_L}.
If
\begin{equation*}
\beta\ge\frac{1}{\nu\left(\nu-\frac{\widetilde{K}_L}{\lambda_1}\right)} \left(\|u_0\|^2_H+ C_L\right)
\end{equation*}
and 
\begin{equation}
\label{cond_mu_iii}
\mu \ge (1+\delta)^2 \left(L+ \frac{C_L}{\nu\left(\nu-\frac{\widetilde{K}_L}{\lambda_1}\right)}\right),
\end{equation}
then
\begin{equation*}
\label{tau_est_ii}
\mathbb{P}(\tau_{R, \beta, \delta, \mu}<+ \infty)
\le 
 \frac{C(1+ \|u_0\|_H^{4(p+1)})}{R^p},
 \end{equation*}
 for any $p\in \left(0, \frac{\nu \lambda_1}{4\widetilde{K}_L}-\frac 12 \right)$,  where $C$ is a positive constant depending on $\lambda_1,p, \nu, K_L, \|f\|_{V^*}$
 and  independent of $R, \beta$ and $u_0$.
\end{itemize}

\end{prop}
\begin{proof}
For the proof one argues similarly as in the proof of \cite[Proposition 4.3]{FZ23}. For the sake of clarity we provide here the proof of statement (i). 

Keeping in mind the definition \eqref{tau_R_beta} of the stopping time $\tau_{R, \beta, \delta, \mu}$,  we introduce the set
\begin{equation*}
A_{R, \beta, \delta, \mu}= \left\{\sup_{r \ge 0} \left[\frac{1}{\nu} \int_0^r \|u(s)\|^2_V\, {\rm d}s +\left(L-\frac{\mu }{(1+\delta)^2}\right)r - \beta\right] \ge R\right\}
\end{equation*}
so that $\mathbb{P}(\tau_{R, \beta, \delta, \mu}< +\infty) \le \mathbb{P}(A_{R, \beta, \delta, \mu})$. 
We write the complementary set of $A_{R, \beta, \delta, \mu}$ as 
\begin{equation*}
A_{R, \beta, \delta, \mu}^c=\left\{ \frac \nu 2  \int_0^r\|u(s)\|^2_V\, {\rm d}s < \frac{\nu^2}2 \left[\left(\frac{\mu}{(1+\delta)^2}-L\right)r+\beta +R\right] \text{ for any }\ r \ge 0
\right\}.
\end{equation*}
If $\mu$ satisfies condition \eqref{cond_mu_i}, 
then 
\begin{equation*}
\frac{\nu^2}2 \left( \frac{\mu}{(1+\delta)^2}-L\right)\ge \frac{C_B}{2}\equiv K_B+ \frac{1}{\nu }\|f\|^2_{V^*}.
\end{equation*}
We choose $\beta$ as in \eqref{cond_beta_i} and set
$
\bar R:=\frac{\nu^2}2 R$. For this choice of parameters we get 
\begin{equation*}
A_{R, \beta, \delta, \mu}^c 
\supseteq 
\left\{ \frac \nu 2 \int_0^r \|u(s)\|^2_V\, {\rm d}s <\left( K_B+ \frac{1}{\nu } \|f\|^2_{V^*} \right)r+ \bar R + \|u_0\|^2_H  \text{ for any } r \ge 0\right\}
\end{equation*}
i.e.
\[
A_{R, \beta, \delta, \mu} \subseteq 
\left\{ \sup_{r\ge 0} \left[ \frac \nu 2 \int_0^r \|u(s)\|^2_V\, {\rm d}s -\left( K_B+ \frac{1}{\nu} \|f\|^2_{V^*}\right)r-\|u_0\|^2_H \right] \ge \bar R\right\} .
\]
From Lemma  \ref{lem_mul_2}  we therefore conclude that 
\begin{align*}
 \mathbb{P}(A_{R, \beta, \delta, \mu}) \le 
 e^{-\frac{\nu\lambda_1}{8K_B}\bar R}.
 \end{align*}
 Since $\mathbb{P}(\tau_{R, \beta, \delta, \mu}< +\infty) \le \mathbb{P}(A_{R, \beta, \delta, \mu})$, 
 keeping in mind the definition of $\bar R$, the thesis follows.
 
The proof of statements (ii) and (iii) is similar and it follows the lines of the proof of \cite[Proposition 4.3]{FZ23}: instead of exploiting Lemma \ref{lem_mul_2} one uses the estimates in probability of Lemmata \ref{lem_mul_3} and \ref{lem_mul_4}, respectively, obtaining a polynomial decay in the parameter $R$.
\end{proof}

We have now all the ingredients to prove Theorem \ref{convergence_mult_case}, which follows as a consequence of Corollary \ref{cor_tau_R_beta} and Proposition \ref{tau_exp}.

\begin{proof}[Proof of Theorem \ref{convergence_mult_case}]
The structure of the proof is the same under Assumptions \ref{assumption-stochastic2} (B), (S) and (L). 
\\
By means of the H\"older and the Young inequalities, invoking Corollary \ref{cor_tau_R_beta} and Lemmata  \ref{lem_mul_q}--\ref{lem_mul_q_U} with $q=4$, \begin{footnote}{Notice that considering $q=4$ in Lemmata \ref{lem_mul_q}-\ref{lem_mul_q_U} requires to impose the condition $1+ \frac{\nu \lambda_1}{2\widetilde{K}_L}>4$, equivalent to $ \nu >\frac{6\widetilde K_L}{\lambda_1}$, when working under Assumption \ref{assumption-stochastic2}(L).}\end{footnote} we infer, for any $R, \beta, \delta>0$ and $0<\mu \le \frac{  \nu }{c_0 h^2 }$, 
\begin{align}
\label{sti_diff_proof}
\mathbb{E} \left[ \|u(t)-U(t)\|^2_H\right]
&= \mathbb{E} \left[\pmb{1}_{(\tau_{R, \beta, \delta, \mu =+\infty ) }} \|u(t)-U(t)\|^2_H\right]+ \mathbb{E} \left[\pmb{1}_{(\tau_{R, \beta, \delta, \mu <+\infty)}} \|u(t)-U(t)\|^2_H\right]
\notag \\
& \le e^{\beta+ R - \frac{\delta \mu}{(\delta +1)^2} t} \|u_0-U_0\|^2_H+ \left( \mathbb{P}(\tau_{R, \beta, \delta, \mu}<+\infty)\right)^{\frac 12} \left(\mathbb{E}[\|u(t)-U(t)\|^4_H \right)^{\frac 12}
\notag\\
& \le
 C\left(1+\|u_0\|^2_H+ \|U_0\|^2_H \right)\left(\left( \mathbb{P}(\tau_{R, \beta, \delta, \mu}<+\infty)\right)^{\frac 12}+ e^{\beta+ R - \frac{\delta\mu}{(\delta +1)^2} t}\right),
\end{align}
where $C$ is a positive constant depending on the structural parameters of the equation (see Lemmata \ref{lem_mul_q}--\ref{lem_mul_q_U} for the explicit dependence according to which case (B), (S) or (L) in Assumption \ref{assumption-stochastic2} is considered).

We now make a suitable choice of the parameters $R, \beta, \delta$ and $\mu$ and use the previous bounds on $\tau_{R, \beta, \delta \mu}$. From the lower bounds on $\mu$ \eqref{cond_mu_bounded}, \eqref{cond_mu_sublinear} and \eqref{cond_mu_linear} we infer the existence of $\delta>0$ such that the lower bounds \eqref{cond_mu_i}, \eqref{cond_mu_ii} and \eqref{cond_mu_iii}, respectively, in Proposition \ref{tau_exp} hold true. 
By also choosing $\beta$ as in Proposition \ref{tau_exp}, for any $R>0$, we obtain the bounds
\begin{equation*}
{P}(\tau_{R, \beta, \delta, \mu} < +\infty)^{\frac 12} \le
\begin{cases}
e^{-CR} & \text{under Assumption \ref{assumption-stochastic2} (B)},
\\
\frac{C\left(1+\|u_0\|^{2(p+1)}_H\right)}{R^{\frac p2}} & \text{ for any $p \in (0, +\infty)$,  under  Assumption \ref{assumption-stochastic2} (S)},
\\
\frac{C\left(1+\|u_0\|^{2(p+1)}_H\right)}{R^{\frac p2}} & \text{ for any $p\in \left(0, \frac{\nu \lambda_1}{4\widetilde{K}_L}-\frac 12 \right)$, \ under  Assumption \ref{assumption-stochastic2} (L)}, 
\end{cases}
\end{equation*}
with 
\begin{equation*}
    C= 
    \begin{cases}
C(\lambda_1, \nu, K_B) & \text{under Assumption \ref{assumption-stochastic2} (B)},
\\
C\left(\lambda_1, p, \nu, K_S, \widetilde{K}_S,\gamma, \|f\|_{V^*}\right) & \text{ under  Assumption \ref{assumption-stochastic2} (S)},
\\
C\left(\lambda_1, p, \nu, K_L, \widetilde{K}_L, \|f\|_{V^*}\right) & \text{  under  Assumption \ref{assumption-stochastic2} (L)},
\end{cases}
\end{equation*}
where we emphasize that the constant $C$ does not depend on $u_0, U_0, R, \beta$ and $t$. Coming back to estimate \eqref{sti_diff_proof}, 
if we select $R= \frac{\delta\mu}{2(1+\delta)^2}t$, for each $t > 0$, we conclude the proof. 
\end{proof}

\subsection{Some partial results for pathwise convergence}
Since $\displaystyle\lim_{t\to+\infty} \EE \|u(t_n)-U(t_n)\|^2_H=0$,  there exists  a subsequence converging $\PX$-a.s. 

We now show that given  any subsequence with $t_n\to+\infty$ as $n\to+\infty$ we have 
$\|u(t_n)-U(t_n)\|^2_H\to 0$, $\PX$-a.s.
Notice that this is not enough to prove that there is pathwise convergence $\|u(t)-U(t)\|^2_H\to 0$ as $t\to+
\infty$.

\begin{coro}
\label{pathwise_mult}
Under the same assumption of Theorem \ref{convergence_mult_case}, for any diverging increasing sequence of times $(t_n)_n$ we have 
 \[
\mathbb{P}\left( \lim_{ n \rightarrow + \infty}\|u(t_n)-U(t_n)\|^2_H =0\right) = 1.
 \]
\end{coro}
\begin{proof}
For any $n\in\mathbb{N}$ we introduce the events 
\[
B_n:= \left\{\|u(t_n)-U(t_n)\|^2_H > \frac{1}{t_n^2}\right\}
\]
and we set
\[
B:=\bigcap_{m=1}^{+\infty}\bigcup_{n=m}^{+\infty}B_n.
\]
We consider the stopping time $\tau_{R, \beta, \delta, \mu}$ introduced in \eqref{tau_R_beta} and write 
\[
\mathbb{P}(B)=\mathbb{P}\left(B \cap (\tau_{R,\beta, \delta, \mu}=+ \infty)\right)+\mathbb{P}\left(B \cap (\tau_{R, \beta, \delta, \mu}<+ \infty)\right).
\]
We now observe that $\mathbb{P}\left(B \cap (\tau_{R, \beta, \delta, \mu}=+ \infty)\right)=0$ for any $R, \beta, \delta>0$ and $0<\mu \le \frac{  \nu }{c_0 h^2}$. In fact, thanks to the Markov inequality and Corollary \ref{cor_tau_R_beta}, for every $n \in \mathbb{N}$ it holds
\begin{align*}
\mathbb{P}\left(B_n \cap (\tau_{R, \beta, \delta, \mu}=+ \infty)\right)
&\le t_n^2 \mathbb{E}\left[\pmb{1}_{\left(\tau_{R,\beta, \delta, \mu}=+ \infty\right)} \|u(t_n)-U(t_n)\|_H^2\right]
\\
&\le e^{R+ \beta} \|u_0-U_0\|^2_H t_n^2 e^{-\frac{\delta\mu}{(\delta+1)^2}t_n}.
\end{align*}
Thus, 
\[
\sum_{n\in\mathbb{N}} \mathbb{P} \left(B_n \cap  (\tau_{R, \beta,\delta,\mu}=+\infty) \right) <+ \infty.
\]
Hence, by the Borel-Cantelli Lemma it follows $\mathbb{P} \left(B \cap  (\tau_{R, \beta, \delta, \mu}=+\infty ) \right)=0$. We infer 
\[
\mathbb{P}(B) = \mathbb{P}\left(B \cap  (\tau_{R, \beta,\delta, \mu}< + \infty) \right)\le \mathbb{P} \left(\tau_{R,\beta, \delta, \mu} < + \infty\right).
\]
We now make a suitable choice of the parameters $R, \beta, \delta$ and $\mu$ and use the bounds on $\tau_{R, \beta, \delta,\mu}$. From the lower bounds on $\mu$ \eqref{cond_mu_bounded}, \eqref{cond_mu_sublinear} and \eqref{cond_mu_linear} we infer the existence of $\delta>0$ such that the lower bounds \eqref{cond_mu_i}, \eqref{cond_mu_ii} and \eqref{cond_mu_iii}, respectively, in Proposition \ref{tau_exp} hold true.
By also choosing $\beta$ as in Proposition \ref{tau_exp}, for any $R>0$, we infer
\begin{equation*}
 \mathbb{P}\left(\tau_{R, \beta, \delta, \mu}<+ \infty\right) \le
\begin{cases}
e^{-CR} & \text{under Assumption \ref{assumption-stochastic2} (B)},
\\
\frac{C}{R^{p}} & \text{ for any $p \in (0, +\infty)$,  under  Assumption \ref{assumption-stochastic2} (S)},
\\
\frac{C}{R^{p}} & \text{ for any $p\in \left(0, \frac{\nu \lambda_1}{4\widetilde{K}_L}-\frac 12 \right)$, \ under  Assumption \ref{assumption-stochastic2} (L)}, 
\end{cases},
\end{equation*}
 where the constant $C$ does not depend on $R$ (and depends on the structural parameters of the equation; see the proof of Theorem \ref{convergence_mult_case} for more details).
Hence, for any $\varepsilon>0$ we can fix $\overline R=\overline R(\varepsilon)>0$ such that $\mathbb{P}(B)<\varepsilon$.
With such choice of $\overline R$, 
 from the continuity from below we can find $m^*>0$ sufficiently large such that 
\[
\mathbb{P}\left( \bigcap_{n=m^*}^{+\infty}B_n^c\right)\ge1-\varepsilon.
\]
Now we observe that 
\[
\left\{ \lim_{ n \rightarrow +\infty} \|u(t_n)-U(t_n)\|_H^2 =0 \right\} \supseteq \bigcap_{n=m^*}^{+\infty}B_n^c,
\]
hence,
\[
\mathbb{P} \left( \lim_{n \rightarrow+ \infty} \|u(t_n)-U(t_n)\|_H^2  =0 \right) \ge \mathbb{P} \left(  \bigcap_{n=m^*}^{+\infty}B_n^c\right)\ge 1-\varepsilon.
\]
By the arbitrariness of $\varepsilon$ we conclude the proof.
\end{proof}

\section{Data assimilation in the stochastic framework: the additive noise case}
\label{sec:add}

Let us now turn to the study of the 2D Navier-Stokes equations driven by an additive noise.  
This is the particular case when the covariance of the noise does not depend on the velocity $u$, i.e.
\begin{equation}
\label{2D-NS-additive-abs}
\begin{cases}
{\rm d}u(t) + \left[\nu Au(t)+B(u(t),u(t))\right]\,{\rm d}t=f \, {\rm d}t+G\,{\rm d}W(t),\qquad\qquad t>0
\\
u(0)=u_0\in H
\end{cases}
\end{equation}
As before, we assume $\nu>0$ and  $ f \in V^*$ independent of time, and we do not repeat these assumptions in the next statements.

 \begin{theorem}
 \label{NS-wp}
Let $G \in L_{HS}(U;H)$. 
Then,  there exists a unique strong solution $u$ to equation \eqref{2D-NS-additive-abs} with $\PP$-a.s. paths in 
\begin{equation*}
\mathcal{C}([0,+\infty), H) \cap L_{loc}^2(0,+\infty;V),
\end{equation*}
that $\mathbb{P}$-a.s. satisfies
  \begin{align*}			
  \left( u(t), \phi\right)_H +  \int_0^t \left( A^{1/2} u(s), A^{1/2} \phi\right)_H\, {\rm d}s + \int_0^t \big\langle B(u(s),\phi),  u(s) \big\rangle \,{\rm d}s
=  \left( u_0, \phi\right)_H + \langle f, \phi\rangle t +\big\langle GW(t), \phi \rangle,
  \end{align*}
for every $t \ge 0$ and every $\phi \in V$. 
\end{theorem}
Mean square estimates for the $H$ and $V$-norm of the solution $u$ are given in Appendix \ref{sec:appB}.

The data assimilation equation in the additive case is now
\begin{equation}
\begin{cases}
\label{data-assi-NS-equation-abs}
{\rm d} U(t)+ \left( \nu A U(t) + B(U(t),U(t)) \right) {\rm d}t  =  G{\rm d} W(t) + f {\rm d}t - \mu \Pi R_{h}(U(t)-u(t))   {\rm d}t    ,
\\
U(0)=U_0 \in H
\end{cases}
\end{equation}
where $u$ is the solution of the Navier-Stokes equation \eqref{2D-NS-additive-abs} and $R_h$, defined in Section \ref{sect:interpolant-operators}, fulfills \eqref{iden-data-approx}.
Notice that equation \eqref{data-assi-NS-equation-abs} admits a unique strong solution with the same regularity as the solution to equation \eqref{2D-NS-additive-abs}. 

In the previous section,  we obtained   convergence in expectation
when the multiplicative coefficient is bounded.
In contrast, in the presence of additive noise, the pathwise behavior can be investigated in a more straightforward manner.
We first establish the convergence result in expectation,
subsequently, we turn to the pathwise analysis.

\subsection{Convergence in expectation}

The convergence in the mean square sense is obtained exactly as in the multiplicative noise setting. The next result can indeed be seen as a particular case of  part i) of Theorem \ref{NS-wp}, considering $L=0$ and $K_B= \|G(u)\|^2_{L_{HS}(U,H)}$.
\begin{theorem}
\label{exp-conv-prop}
Let $G\in L_{HS}(U;H)$ and $R_h$ satisfy assumption \eqref{iden-data-approx}.
Let $u$ be the solution to the Navier-Stokes  equation \eqref{2D-NS-additive-abs} and  $U$ the solution to the  data assimilation equation \eqref{data-assi-NS-equation-abs}. If $\mu$ and $h$ satisfy
\begin{equation}
\label{cond_mu_add}
\frac{2}{\nu^2}\left(\|G\|^2_{L_{HS}(U,H)}+\frac{\|f\|^2_{V^*}}{\nu}\right) < \mu \le \frac{  \nu }{c_0 h^2 }, 
\end{equation}
then
$\mathbb{E}\left[\|u(t)-U(t)\|_H^2\right] \rightarrow 0$ exponentially fast as $t \rightarrow + \infty$.
\end{theorem}

\subsection{Pathwise convergence}
Let us now focus on the pathwise analysis, meaning that convergence results hold $\mathbb P$-almost surely.
Our main result reads as follows.
\begin{theorem}
\label{pathwise_data_ass}
Let $G\in L_{HS}(U;H)$ and let $R_h$ satisfy assumption \eqref{iden-data-approx}.
Let $u$ be the solution to the Navier-Stokes  equation \eqref{2D-NS-additive-abs} and  $U$ the solution to the  data assimilation equation \eqref{data-assi-NS-equation-abs}. If $\mu$ and $h$ satisfy
\begin{equation}
\label{cond_mu}
\frac{2}{\nu^2}\left(\|G\|^2_{L_{HS}(U,H)}+\frac{\|f\|^2_{V^*}}{\nu}\right) < \mu \le \frac{  \nu }{c_0 h^2 }, 
\end{equation}
then
$\|u(t)-U(t)\|_H \rightarrow 0$ exponentially fast as $t \rightarrow + \infty$, $\mathbb{P}$-a.s.
\end{theorem}

The proof of the above result is based on the following preliminary lemma.

\begin{lemma}
\label{data-estimate-1}
Let  $G\in L_{HS}(U;H)$ and let $R_h$ satisfy assumption \eqref{iden-data-approx}.
Let $u$ be the solution to the Navier-Stokes  equation \eqref{2D-NS-additive-abs} and  $U$ the solution to the  data assimilation equation \eqref{data-assi-NS-equation-abs}. If $\mu$ and $h$ satisfy
\begin{equation}
\label{condi-1-mu}
0<\mu \le \frac{  \nu }{c_0 h^2 } 
\end{equation}
 where $c_0$ is the constant that appears in estimate \eqref{iden-data-approx},
$\PP$-a.s.  we have
\begin{align}
\label{error-u-U}
\left\| u(t)-U(t) \right\|_{H}^2 &\leq \left\|u_0 - U_0 \right\|_{H}^2 \quad  \exp \left(- \mu t + \frac{1}{\nu} \displaystyle \int_{0}^t  \left\| u(s) \right\|_{V}^2 ds   \right), \qquad \text{ for any }t>0.
\end{align} 
\end{lemma}

\begin{proof}
The proof is very similar to that of Lemma \ref{data-estimate-2}. But now the difference $r=U-u$ 
fulfills an equation without noise and  with random coefficients.

Given $u$ and $U$ satisfying the equations \eqref{2D-NS-additive-abs} and \eqref{data-assi-NS-equation-abs}, respectively,  we obtain the evolution of the difference $r(t)=U(t)-u(t)$
\begin{equation*}
\begin{cases}
\frac{{\rm d}}{dt} r(t) + \nu A r(t) +  B(U(t),U(t)) - B(u(t),u(t))   = - \mu \Pi R_h(r(t))  \\ 
r(0)=U_0 - u_0\in H
\end{cases}
\end{equation*}
There is no longer  the noise term; 
the pathwise estimate is easily obtained. Hence, from now on we work pathwise, i.e. the path is fixed.
Taking the scalar product with $r(t)$ we obtain the equality
\begin{align*}
\frac{{\rm d}}{{\rm d}t} \left\| r(t) \right\|_{H}^2 + 2 \nu \left\|  r(t) \right\|_{V}^2  &= -2 \langle r(t),B(r(t),u(t)) \rangle   - 2 \mu \langle R_h(r(t)),r(t) \rangle,
\end{align*}
which holds for every $t \ge 0$. Proceeding similarly to the proof of Lemma \ref{data-estimate-2} we infer 
\begin{align*}
\frac{{\rm d}}{{\rm d}t}  \left\| r(t) \right\|_{H}^2 + \nu \left\| r(t) \right\|_{V}^2  + \mu  \left\| r(t) \right\|_{H}^2 
& \leq \mu c_0 h^2 \left\| r(t) \right\|_{V}^2  + \frac{1}{\nu} \left\| r(t) \right\|^2_{H} \left\|  u(t) \right\|_{V}^2 
\end{align*}
which yields
\begin{align*}
 \frac{{\rm d}}{{\rm d}t} \left\| r(t) \right\|_{H}^2 + \left(\nu - \mu c_0 h^2 \right) \left\|  r(t) \right\|_{V}^2  + 
 \left( \mu  - \frac{1}{\nu} \left\|  u(t) \right\|_{V}^2 \right) \left\|r(t) \right\|_{H}^2  \leq 0.
\end{align*}
 By \eqref{condi-1-mu} we have  
$\nu - \mu c_0 h^2 \ge 0$; hence 
\[
\frac{{\rm d}}{{\rm d}t} \left\| r(t) \right\|_{H}^2  +
 \left( \mu  - \frac{1}{\nu} \left\|  u(t) \right\|_{V}^2 \right) \left\|r(t) \right\|_{H}^2  \le 0.
\]
Using the Gronwall inequality, we infer
\[
    \left\| r (t) \right\|_{H}^2 
    \le \left\|r(0) \right\|_{H}^2 \ \exp \left(  - \mu  t  + \frac{1}{\nu} \displaystyle \int_{0}^t\left\|  u(s) \right\|_{V}^2 \,{\rm d}s   \right),
\]
for any time $t>0$.
\end{proof}

We have now all the ingredients to prove Theorem \ref{pathwise_data_ass}.
\begin{proof}[Proof of Theorem \ref{pathwise_data_ass}]
We want to consider the limit as $t\to+\infty$  in the right hand side of  \eqref{error-u-U}.

Let $G$ satisfy the Assumptions of Theorem \ref{NS-wp}.
From \eqref{stima-certa}  we know that for any initial velocity $u_0\in H$, 
\begin{equation}\label{stima-limsup-normaV}
\limsup_{t \rightarrow +\infty}\frac 1{\nu t}\int_0^t\|u(s)\|^2_V\,{\rm d}s\le \frac 2{\nu^2}\left(\|G\|^2_{L_{HS}(U,H)} + \frac{\|f\|^2_{V^*}}{\nu}\right)
\end{equation}
with probability one. From now on we work pathwise. 
From assumption \eqref{cond_mu} there exists a positive constant $\varepsilon$ such that
\[
 \mu-\frac{2}{\nu^2}\left(\|G\|^2_{L_{HS}(U,H)}+\frac{\|f\|^2_{V^*}}{\nu}\right) - \varepsilon>0.
\]
Now we write the exponent in the right hand side of  \eqref{error-u-U} as follows
\begin{multline*}
-t \mu + \frac{1}{\nu }  \int_{0}^t  \| u(s) \|_{V}^2 {\rm d}s  
=
-t \left[\mu - \frac 2{\nu^2} \left(\|G\|^2_{L_{HS}(U,H)} + \frac{\|f\|^2_{V^*}}{\nu} \right) -\varepsilon   \right]
\\+t \left[\frac{1}{\nu t}  \int_{0}^t  \left\| u(s) \right\|_{V}^2 ds -  \frac 2{\nu^2} \left(\|G\|^2_{L_{HS}(U,H)} +  \frac{\|f\|^2_{V^*}}{\nu }\right) -\varepsilon \right].
\end{multline*}
Thanks to \eqref{stima-limsup-normaV} we get that
\[
\lim_{t\to +\infty} e^{t \left[\frac{1}{\nu t}  \int_{0}^t  \left\| u(s) \right\|_{V}^2 ds -  \frac 2{\nu^2} \left(\|G\|^2_{L_{HS}(U,H)} + \frac{\|f\|^2_{V^*}}{\nu }\right) -\varepsilon \right]}=0
\]
for any $\varepsilon>0$. Moreover
\[
\lim_{t\to +\infty} e^{-t \left[\mu - \frac 2{\nu^2} \left(\|G\|^2_{L_{HS}(U,H)} + \frac{\|f\|^2_{V^*}}{\nu }\right) -\varepsilon   \right]}=0.
\]
Therefore
\[
\lim_{t\to +\infty} e^{-t \mu + \frac{1}{\nu }  \int_{0}^t  \| u(s) \|_{V}^2 {\rm d}s  }=0
\]
and by \eqref{error-u-U}  also $u-U$ vanishes for large times. 
In particular,  we have obtained that there exists a time $\tau>0$ such that for any $t>\tau$ we have
\[
\|u(t)-U(t)\|_H^2 \le 
\|u_0-U_0\|^2_H
  \exp\left(-\left(\mu-\tfrac{2}{\nu^2}\big(\|G\|^2_{L_{HS}(U,H)}+\tfrac 1\nu \|f\|^2_{V^*}\big)-\varepsilon \right)t\right)
\]
showing the exponential convergence rate.
\end{proof}

\begin{remark} 
Assuming that the Navier–Stokes equation \eqref{2D-NS-additive-abs} admits a unique (and hence ergodic) invariant measure, one can obtain an alternative proof of Theorem \ref{pathwise_data_ass} by applying the Birkhoff Ergodic Theorem. We point out, however, that although this approach leads to the same conclusion as Theorem \ref{pathwise_data_ass}, it requires more restrictive assumptions on the operator $G$, since they must guarantee the existence and uniqueness of an invariant measure rather than merely the well-posedness of the equation. For completeness, we present this argument in Appendix \ref{sec:appC}.
\end{remark}

\begin{remark} Notice that, by taking $G=0$ in equation \eqref{2D-NS-additive-abs}, we recover (modulo a constant) the same condition on the nudging parameter $\mu$ given in \cite{azouani2014continuous} (see Theorem \ref{AOT_thm}) for the deterministic setting, in the case of a time independent forcing term $f$ (compare condition \eqref{cond_mu_det} with condition \eqref{cond_mu_add}).

Thus, in the presence of additive stochastic noise, our analysis shows that convergence still holds almost surely, with an exponential rate, provided that the nudging parameter $\mu$ is chosen large enough to balance both the intensity of the deterministic forcing and that of the stochastic perturbation.
This demonstrates that our result can be viewed as a natural generalization of the deterministic result in \cite{azouani2014continuous}, extending it to the case of stochastic forcing.
\end{remark}

\textbf{Acknowldegments.} 
B.F. and M.Z. are members of Gruppo Nazionale per l’Analisi Matematica, la Probabilità e le loro Applicazioni (GNAMPA), Istituto Nazionale di Alta Matematica (INdAM).
\\
B.F. and H.B.  gratefully  acknowledge the financial support from the "Dipartimento di Eccellenza" program and the CAMRisk centre at the Department of Economics and Management of the University of Pavia. 
Indeed, this work was initiated while H.B. was visiting the Department of Economics at the University of Pavia in Spring 2025.
\\
H.B.  gratefully acknowledges the  support of the National Science Foundation under Grant No. DMS-2147189.
H.B. was also in residence at the Simons Laufer Mathematical Sciences Institute (SLMath) in Berkeley, California, during Fall 2025, supported by the National Science Foundation under Grant No. DMS-1928930.
\\
M.Z. gratefully  acknowledges the financial support of the project  ``Prot. P2022TX4FE\_02 -  Stochastic particle-based anomalous reaction-diffusion models with heterogeneous interaction for radiation therapy'' financed by the European Union - Next Generation EU, Missione 4-Componente 1-CUP: D53D23018970001.

\appendix

\section{Apriori estimates: the multiplicative noise case}
As before, everywhere we assume that $f\in V^*$ and the initial data are in $H$. We only specify the assumptions on the covariance of the noise.

\label{sec:appA}
\subsection{Estimates in expected value}
\begin{lemma}
\label{lem_mul_2_0}
Let Assumption \ref{assumption-stochastic1} be in force.
Let $u$ denote the solution to the Navier-Stokes equation \eqref{NS_abs_mult}. 
\begin{itemize}
    \item [(i)] If Assumption \ref{assumption-stochastic2}(B) is in force, then for every $t \ge 0$
\begin{equation*}
\mathbb{E}\left[\|u(t)\|^2_H\right]  + \nu \int_0^t \mathbb{E}\left[\|u(s)\|^2_V\right]\, {\rm d}s \le \|u_0\|^2_H + \left(K_B + \frac{1}{\nu} \|f\|^2_{V^*} \right)t.
\end{equation*}
 \item [(ii)] If Assumption \ref{assumption-stochastic2}(S) is in force, then for every $t \ge 0$
\[
\mathbb{E}\left[\|u(t)\|^2_H\right]  + \nu \int_0^t \mathbb{E}\left[\|u(s)\|^2_V\right]\, {\rm d}s 
\le 
\|u_0\|^2_H + \left(K_S + 
(1-\gamma)  \Big( \frac{2\gamma} {\lambda_1 \nu}\Big)^{\frac{\gamma}{1-\gamma}} \widetilde{K_S}^{\frac{1}{1-\gamma}}
 +
\frac{2}{\nu}\|f\|^2_{V^*} \right)t.
\]
  \item [(iii)] If Assumption \ref{assumption-stochastic2}(L) is in force and $\nu > \frac{\widetilde K_L}{\lambda_1}$, then for every $t \ge 0$
\begin{equation*}
\mathbb{E}\left[\|u(t)\|^2_H\right]  + \left( \nu - \frac{\widetilde K_L}{\lambda_1} \right) \int_0^t \mathbb{E}\left[\|u(s)\|^2_V\right]\, {\rm d}s \le \|u_0\|^2_H + \left(K_L + \frac{1}{\nu} \|f\|^2_{V^*} \right)t.
\end{equation*}
\end{itemize}
\end{lemma}
\begin{proof}
We apply the It\^o formula to the functional $\|\cdot\|^2_H$ and, exploiting \eqref{B}, we infer, $\mathbb{P}$-a.s., for any $t\ge0$,
\begin{multline}
\label{sti1_app_mult}
\|u(t)\|^2_H + 2 \nu\int_0^t \|u(s)\|^2_V\, {\rm d}s  = \|u_0\|_H^2+ \int_0^t\|G(u(s))\|^2_{L_{HS}(U,H)}{\rm d}s 
\\
 +2 \int_0^t\langle u(s), G(u(s))\, {\rm d}W(s)\rangle + 2 \int_0^t \langle u(s), f\rangle\, {\rm d}s.
\end{multline}
The Young and Poincar\'e inequalities yield, for  arbitrary $\eta>0$,
\begin{equation}
\label{stima_F}
   2\langle u, f\rangle \le 2 \|u\|_V\|f\|_{V^*}\le \eta \nu \|u\|^2_V + \frac{1}{\eta \nu} \|f\|^2_{V^*} .
\end{equation}
Thus, by taking the expected value on both sides of \eqref{sti1_app_mult}, recalling that the stochastic integral is a (local) martingale, we infer 
\begin{multline}
\label{sti1_app_mult2}
\mathbb{E}\left[\|u(t)\|^2_H\right] + (2-\eta) \nu\int_0^t \mathbb{E}\left[\|u(s)\|^2_V\right]\, {\rm d}s  
\le  \|u_0\|_H^2+ \int_0^t\mathbb{E}\left[\|G(u(s))\|^2_{L_{HS}(U,H)}\right]{\rm d}s + \frac{1}{\eta \nu} \|f\|^2_{V^*}t.
\end{multline}
Assumption \ref{assumption-stochastic2}(B) and the choice  $\eta=1$ in estimate \eqref{stima_F} immediately yields statement (i).
Assumption \ref{assumption-stochastic2}(S), the Poincar\'e  and Young inequalities yield, for  arbitrary $\varepsilon>0$,
\[
\|G(u)\|^2_{L_{HS}(U,H)}\le K_S + \frac{\widetilde K_S}{\lambda_1^\gamma} \|u\|^{2\gamma}_V
\le 
K_S +  (1-\gamma)\left(\frac{\gamma}{\varepsilon\lambda_1 \nu}\right)^{\frac{\gamma}{1-\gamma}} \widetilde{K_S}^{\frac{1}{1-\gamma}} + \varepsilon \nu \|u\|^2_V.
\]
With the choice $\eta=\varepsilon=\frac 12$ we infer the result in (ii).

Assumption \ref{assumption-stochastic2}(L) and the Poincar\'e  inequality yield,
\[
\|G(u)\|^2_{L_{HS}(U,H)}\le K_L +\widetilde K_L \|u\|^{2}_H
\le K_S + \frac{\widetilde{K}_L}{\lambda_1} \|u\|^2_V,
\]
from which statement (iii) follows by taking $\eta=1$ and assuming $\nu > \frac{\widetilde{K}_L}{\lambda_1}$. This concludes the proof.
\end{proof}

\begin{lemma}
\label{lem_mul_2_new}
Let Assumption \ref{assumption-stochastic1} be in force.
Let $u$ denote the solution to the Navier-Stokes equation \eqref{NS_abs_mult}. 
\begin{itemize}
    \item [(i)] If Assumption \ref{assumption-stochastic2}(B) is in force, then 
\begin{equation*}
\sup_{t \ge 0}\mathbb{E}\left[\|u(t)\|^2_H\right]   \le \|u_0\|^2_H + \frac{1}{\nu\lambda_1}\left(K_B + \frac{1}{\nu} \|f\|^2_{V^*} \right).
\end{equation*}
 \item [(ii)] If Assumption \ref{assumption-stochastic2}(S) is in force, then 
\[
\sup_{t \ge 0}\mathbb{E}\left[\|u(t)\|^2_H\right]   \le \|u_0\|^2_H + \frac{1}{\nu\lambda_1}\left(K_S + 
 (1-\gamma)\left(\frac{2\gamma}{\lambda_1 \nu}\right)^{\frac{\gamma}{1-\gamma}} \widetilde{K_S}^{\frac{1}{1-\gamma}}
 +
\frac{2}{\nu}\|f\|^2_{V^*}  \right).
\]
  \item [(iii)] If Assumption \ref{assumption-stochastic2}(L) is in force and $\nu > \frac{\widetilde K_L}{\lambda_1}$, then 
\begin{equation*}
\sup_{t \ge 0}\mathbb{E}\left[\|u(t)\|^2_H\right]  \le \|u_0\|^2_H + \frac{1}{\nu \lambda_1-\widetilde K_L}\left(K_L + \frac{1}{\nu} \|f\|^2_{V^*} \right).
\end{equation*}
\end{itemize}
\end{lemma}
\begin{proof}
We start from the estimate \eqref{sti1_app_mult}; then the term with the Hilbert-Schmidt norm of $G$ is estimated according to the three Assumptions \ref{assumption-stochastic2}. Thus we get
\[
\mathbb{E}\left[\|u(t)\|^2_H\right]  + a \lambda_1 \int_0^t \mathbb{E}\left[\|u(s)\|^2_H\right]\, {\rm d}s \le \|u_0\|^2_H + b t,\qquad \forall t\ge 0,
\]
where 
\begin{equation*}
a=
    \begin{cases}
    \nu & \text{under Assumptions \ref{assumption-stochastic2}(B)-(S)},
    \\
    \nu -\frac{\widetilde{K_L}}{\lambda_1} & \text{under Assumptions \ref{assumption-stochastic2}(L)},
    \end{cases}
\end{equation*}
and 
\begin{equation*}
b=
    \begin{cases}
    K_B + \frac{1}{\nu} \|f\|^2_{V^*} & \text{under Assumptions \ref{assumption-stochastic2}(B)},
    \\
    K_S + 
(1-\gamma)\left(\frac{2\gamma}{\lambda_1 \nu}\right)^{\frac{\gamma}{1-\gamma}} \widetilde{K_S}^{\frac{1}{1-\gamma}}
 +
\frac{2}{\nu}\|f\|^2_{V^*} & \text{under Assumptions \ref{assumption-stochastic2}(S)},
    \\
    K_L + \frac{1}{\nu} \|f\|^2_{V^*}  & \text{under Assumptions \ref{assumption-stochastic2}(L)}.
    \end{cases}
\end{equation*}
The Gronwall lemma yields
\[
\mathbb{E}\left[\|u(t)\|^2_H\right]
\le
e^{-  a\lambda_1  t  } \|u_0\|^2_H+\frac{b}{a\lambda_1}
\le \|u_0\|^2_H+\frac{b}{a\lambda_1},
\]
for any time $t\ge 0$, which proves the result.
\end{proof}

Next we prove similar estimates for any power $q>2$.

\begin{lemma}
\label{lem_mul_q}
Let Assumption \ref{assumption-stochastic1} be in force.
Let $u$ denote the solution to the Navier-Stokes equation \eqref{NS_abs_mult}. 
\begin{itemize}
    \item [(i)] If Assumption \ref{assumption-stochastic2}(B) is in force, then 
\begin{equation*}
\sup_{t\ge 0}\mathbb{E}\left[\|u(t)\|^q_H\right]  
\le 
\|u_0\|^q_H + \frac{4C_1}{q\nu\lambda_1}, \quad \text{for all} \ q\in (2, +\infty),
\end{equation*}
where $C_1=\frac{8}{q\nu\lambda_1}\left(\left(K_B \frac{q(q-1)}{2}\right)^{\frac q2}+  \left(\frac{2q}{ \nu} \right)^{\frac q2} \|f\|^q_{V^*} \right)$.
 \item [(ii)] If Assumption \ref{assumption-stochastic2}(S) is in force, then 
   \begin{equation*}
\sup_{t\ge 0}\mathbb{E}\left[\|u(t)\|^q_H\right]  \le \|u_0\|^q_H + \frac{4C_2}{q\nu\lambda_1}, \quad \text{for all} \ q \in (2, +\infty),
\end{equation*}
where $C_2=\frac{12}{q\nu\lambda_1}\left(\left(K_S \frac{q(q-1)}{2}\right)^{\frac q2} + \left(\widetilde K_S \frac{q(q-1)}{2}\right)^{\frac{q}{2(1-\gamma)}}  +  \left(\frac{2q}{ \nu} \right)^{\frac q2} \|f\|^q_{V^*} \right)$.
  \item [(iii)] If Assumption \ref{assumption-stochastic2}(L) is in force and $\nu > \frac{2\widetilde K_L}{\lambda_1}$, then 
\begin{equation*}
\sup_{t\ge 0}\mathbb{E}\left[\|u(t)\|^q_H\right]  \le \|u_0\|^q_H + \frac{C_3}{\frac{q\nu\lambda_1}{4}-\frac{q(q-1)\widetilde{K}_L}{2}}, \quad \text{for all} \ q\in \left(2, 1+ \frac{\nu\lambda_1}{2\widetilde K_L}\right),
\end{equation*}
where $C_3=\frac{8}{q\nu\lambda_1}\left(\left(K_L \frac{q(q-1)}{2}\right)^{\frac q2}+  \left(\frac{2q}{ \nu} \right)^{\frac q2} \|f\|^q_{V^*} \right)$.
\end{itemize}
\end{lemma}
\begin{proof}
We apply the It\^o formula to the functional $\|\cdot\|^q_H$ for $q >  2$. 
Exploiting \eqref{B}, we infer $\mathbb{P}$-a.s., for every $t \ge 0$,
\begin{multline}
\label{eq_q_u}
{\rm d}\|u(t)\|^q_H + q \nu \|u(t)\|_H^{q-2}\|u(t)\|^2_V\, {\rm d}t
\le \frac{q(q-1)}{2}\|u(t)\|^{q-2}_H\|G(u(t))\|^2_{L_{HS}(U,H)}\, {\rm d}t 
\\
+q\|u(t)\|^{q-2}_H\langle u(t), G(u(t)){\rm d}W(t)\rangle + q\|u(t)\|_H^{q-2}\langle u(t), f\rangle\,{\rm d}t.
\end{multline}
Using repeatedly the Young inequality, for any $\alpha, \varepsilon, \eta>0$, we estimate 
\begin{align}
\label{stima_f_q}
    q\|u\|_H^{q-2}\langle u, f\rangle 
    &\le q \|u\|_H^{q-2} \left( \eta \nu \|u\|^2_V + \frac{1}{\eta\nu}\|f\|^2_{V^*}\right)
    \\\notag
    & \le q\eta \nu \|u\|^{q-2}_H \|u\|^{2}_V+ \frac{\alpha}{\varepsilon}\left(\frac{q}{\eta \nu} \right)^{\frac q2} \|f\|^q_{V^*} + \frac{\varepsilon}{\alpha}\|u\|^q_H.
\end{align}
Thus, from the Poincar\'e inequality we infer $\mathbb{P}$-a.s., for every $t \ge 0$,
\begin{multline*}
{\rm d}\|u(t)\|^q_H + \left(q \nu\lambda_1(1-\eta)-\frac{\varepsilon}{\alpha}\right) \|u(t)\|_H^{q}\, {\rm d}t
\le \frac{q(q-1)}{2}\|u(t)\|^{q-2}_H\|G(u(t))\|^2_{L_{HS}(U,H)}\, {\rm d}t 
\\
+q\|u(t)\|^{q-2}_H\langle u(t), G(u(t)){\rm d}W(t)\rangle +  \frac{\alpha}{\varepsilon}\left(\frac{q}{\eta \nu} \right)^{\frac q2} \|f\|^q_{V^*} \,{\rm d}t.
\end{multline*}
Taking the expected value on both sides of the above inequality, since the stochastic integral is a (local) martingale, we get 
\begin{multline*}
  \frac{{\rm d}}{{\rm d}t}\mathbb{E}\left[\|u(t)\|^q_H\right] + \left(q \nu\lambda_1(1-\eta)-\frac{\varepsilon}{\alpha}\right) \mathbb{E}\left[\|u(t)\|_H^{q}\right]\,
\\
\le \frac{q(q-1)}{2}\mathbb{E}\left[\|u(t)\|^{q-2}_H\|G(u(t))\|^2_{L_{HS}(U,H)}\right]
+  \frac{\alpha}{\varepsilon}\left(\frac{q}{\eta \nu} \right)^{\frac q2} \|f\|^q_{V^*}.  
\end{multline*}
Now we estimate the second order term in the above expression, according to the different assumptions made on the noise.
\begin{itemize}
    \item [(i)] If Assumption \ref{assumption-stochastic2}(B) is in force, then by the Young inequality we infer, for any $\beta, \varepsilon>0$
\begin{equation*}
 \frac{q(q-1)}{2}\|u\|^{q-2}_H\|G(u)\|^2_{L_{HS}(U,H)} \le K_B \frac{q(q-1)}{2}\|u\|^{q-2}_H \le \frac{\beta}{\varepsilon}\left(K_B \frac{q(q-1)}{2}\right)^{\frac q2} +\frac{\varepsilon}{\beta}\|u\|^q_H.
\end{equation*}
By selecting $\eta=\frac{1}{2}$, $\alpha=\beta=8$, $\varepsilon=q\nu\lambda_1$, we obtain the estimate 
\begin{equation*}
  \frac{{\rm d}}{{\rm d}t}\mathbb{E}\left[\|u(t)\|^q_H\right] + \frac{q \nu\lambda_1}{4}  \mathbb{E}\left[\|u(t)\|_H^{q}\right]
\\ 
\le \frac{8}{q\nu\lambda_1}\left(\left(K_B \frac{q(q-1)}{2}\right)^{\frac q2}+  \left(\frac{2q}{ \nu} \right)^{\frac q2} \|f\|^q_{V^*} \right).  
\end{equation*}
By the Gronwall lemma we infer 
\[
\mathbb{E}\left[ \|u(t)\|^q_H\right]\le \|u_0\|^q_H e^{-\frac{q \nu \lambda_1}4 t}+ \frac{4C_1}{q \nu \lambda_1} (1- e^{\frac{-q \nu \lambda_1}{4}t}),
\]
which proves statement (i).
\item [(ii)] If Assumption \ref{assumption-stochastic2}(S) is in force, then by the Young inequality we infer, for any $\beta, \delta, \varepsilon>0$
\begin{multline*}
 \frac{q(q-1)}{2}\|u\|^{q-2}_H\|G(u)\|^2_{L_{HS}(U,H)} 
 \le  \frac{q(q-1)}{2}\|u\|^{q-2}_H\left( K_S + \widetilde K_S\|u\|_H^{2\gamma}\right) 
 \\
 \le \frac{\beta}{\varepsilon}\left(K_S \frac{q(q-1)}{2}\right)^{\frac q2} +\frac{\varepsilon}{\beta}\|u\|^q_H + \frac{\delta}{\varepsilon}\left(\widetilde K_S \frac{q(q-1)}{2}\right)^{\frac{q}{2(1-\gamma)}} + \frac{\varepsilon}{\delta}\|u\|^q_H.
\end{multline*}
By selecting $\eta=\frac{1}{2}$, $\alpha=\beta=\delta=12$, $\varepsilon=q\nu\lambda_1$, we obtain the estimate 
\begin{equation*}
  \frac{{\rm d}}{{\rm d}t}\mathbb{E}\left[\|u(t)\|^q_H\right] + \frac{q \nu\lambda_1}{4}  \mathbb{E}\left[\|u(t)\|_H^{q}\right]
\\
\le \frac{12}{q\nu\lambda_1}\left(\left(K_S \frac{q(q-1)}{2}\right)^{\frac q2} + \left(\widetilde K_S \frac{q(q-1)}{2}\right)^{\frac{q}{2(1-\gamma)}}  +  \left(\frac{2q}{\nu} \right)^{\frac q2} \|f\|^q_{V^*} \right).  
\end{equation*}
By the Gronwall lemma we infer 
\[
\mathbb{E}\left[ \|u(t)\|^q_H\right]\le \|u_0\|^q_H e^{-\frac{q \nu \lambda_1}4 t}+ \frac{4C_2}{q \nu \lambda_1} (1- e^{\frac{-q \nu \lambda_1}{4}t}),
\]
which proves statement (ii).
\item [(iii)] If Assumption \ref{assumption-stochastic2}(L) is in force, then by the Young inequality we infer, for any $\beta, \varepsilon>0$
\begin{align*}
 \frac{q(q-1)}{2}\|u\|^{q-2}_H\|G(u)\|^2_{L_{HS}(U,H)} 
 &\le  \frac{q(q-1)}{2}\|u\|^{q-2}_H\left( K_L + \widetilde K_L\|u\|_H^{2}\right) 
 \\
 &\le \frac{\beta}{\varepsilon}\left(K_L \frac{q(q-1)}{2}\right)^{\frac q2} + \frac{\varepsilon}{\beta}\|u\|^q_H + \widetilde{K}_L \frac{q(q-1)}{2}\|u\|^q_H.
\end{align*}
By selecting $\eta=\frac{1}{2}$, $\alpha=\beta=8$, $\varepsilon=q\nu\lambda_1$, we obtain the estimate 
\begin{equation*}
  {\rm d}\mathbb{E}\left[\|u(t)\|^q_H\right] + \left(\frac{q \nu\lambda_1}{4}-\widetilde{K}_L \frac{q(q-1)}{2} \right) \mathbb{E}\left[\|u(t)\|_H^{q}\right]\, {\rm d}t
\\
\le \frac{8}{q\nu\lambda_1}\left(\left(K_L \frac{q(q-1)}{2}\right)^{\frac q2}+  \left(\frac{2q}{ \nu} \right)^{\frac q2} \|f\|^q_{V^*} \right)\,{\rm d}t.  
\end{equation*}
If $\nu > \frac{2\widetilde K_L}{\lambda_1}$ and $q\in \left(2, 1+ \frac{\nu\lambda_1}{2\widetilde K_L}\right)$, by the Gronwall lemma we infer 
\[
\mathbb{E}\left[ \|u(t)\|^q_H\right]\le \|u_0\|^q_H e^{- \left(\frac{q \nu\lambda_1}{4}-\widetilde{K}_L \frac{q(q-1)}{2} \right) t}+ \frac{C_3}{\frac{q\nu\lambda_1}{4}-\frac{q(q-1)\widetilde{K}_L}{2}} (1- e^{- \left(\frac{q \nu\lambda_1}{4}-\widetilde{K}_L \frac{q(q-1)}{2} \right)t}),
\]
which proves statement (iii).
\end{itemize}
\end{proof}

Next we prove similar estimates for the solution of the data assimilation equation \eqref{data-assi-NS-equation-abs_mult}.
\begin{lemma}
\label{lem_mul_q_U}
Let Assumption \ref{assumption-stochastic1} be in force.
Let $U$ denote the solution to the Navier-Stokes equation \eqref{data-assi-NS-equation-abs_mult}. 
\begin{itemize}
    \item [(i)] If Assumption \ref{assumption-stochastic2}(B) is in force, then 
\begin{equation*}
\sup_{t\ge 0}\mathbb{E}\left[\|U(t)\|^q_H\right]  \le C_1(1+ \|u_0\|^q_H + \|U_0\|^q_H), \quad \text{for all} \ q\in [2, +\infty),
\end{equation*}
where $C_1$ is a positive constant depending on $q, \nu, K_B, \lambda_1, \|f\|_{V^*}$.
 \item [(ii)] If Assumption \ref{assumption-stochastic2}(S) is in force, then 
\begin{equation*}
\sup_{t\ge 0}\mathbb{E}\left[\|U(t)\|^q_H\right]  \le C_2(1+ \|u_0\|^q_H + \|U_0\|^q_H), \quad \text{for all} \ q \in [2, +\infty),
\end{equation*}
where $C_2$ is a positive constant depending on $q, \nu, K_S, \widetilde{K}_S, \gamma, \lambda_1, \|f\|_{V^*}$.
  \item [(iii)] If Assumption \ref{assumption-stochastic2}(L) is in force and $\nu > \frac{2\widetilde K_L}{\lambda_1}$, then 
\begin{equation*}
\sup_{t\ge 0}\mathbb{E}\left[\|U(t)\|^q_H\right]  \le C_3(1+ \|u_0\|^q_H + \|U_0\|^q_H), \quad \text{for all} \ q\in \left[2, 1+ \frac{\nu\lambda_1}{2\widetilde K_L}\right),
\end{equation*}
where $C_3$ is a positive constant depending on $q, \nu, K_L, \lambda_1, \|f\|_{V^*}$.
\end{itemize}
\end{lemma}
\begin{proof}
One argues as in the proof of \cite[Lemma A.3]{FZ23}. We sketch the proof for the case $q>2$.
\\
By applying the It\^o formula to $\|U\|^q_H$ one obtains an inequality which is of the form \eqref{eq_q_u} but with the addition, on the right hand side, of the term 
\[
-\mu \|U(t)\|_H^{q-2}\langle R_h(U(t)-u(t)),U(t)\rangle.
\]
Thanks to the Cauchy-Schwartz and the Young inequalities, for any $\delta, \varepsilon>0$ one can estimate 
\begin{align*}
-\mu q\|U\|^{q-2}_H \langle R_h(U-u,U\rangle 
&\le \mu q\|U\|^{q-2}_H \left(- \|R_hU\|^2_H + \|u\|_H\|U\|_H\right) 
\\
&\le \mu q\|U\|^{q-2}_H \left( \frac{1}{\varepsilon}\|u\|_H^2+\varepsilon \|U\|^2_H\right) 
\\
&\le \varepsilon \mu q\|U\|^{q}_H + \left(\mu q\right)^{\frac{q}{q-2}}\delta\|U\|^q_H + \frac{1}{\delta \varepsilon^{\frac q2}}\|u\|^q_H
\end{align*}
By choosing $\varepsilon$ and $\delta$ suitable small, the first two terms can be reabsorbed on the left hand side of the inequality for $\|U\|^q_H$, thus concluding the estimates as in the proof of Lemma \ref{lem_mul_q}; whereas, for the last term involving $u$, one uses the estimates of Lemma \ref{lem_mul_q} itself. This concludes the proof.
\end{proof}

\subsection{Estimates in probability}

\begin{lemma}
\label{lem_mul_2}
Let Assumptions \ref{assumption-stochastic1} and \ref{assumption-stochastic2}(B) be in force.
Let $u$ denote the solution to the Navier-Stokes equation \eqref{NS_abs_mult}. Then
\begin{equation*}
\mathbb{P}\left( \sup_{t\ge T}\left[ \|u(t)\|^2_H + \frac{\nu}{2} \int_0^t \|u(s)\|^2_V \, {\rm d}s -\|u_0\|^2_H -\left(K_B+ \frac{\|f\|^2_{V^*}}{\nu}\right)t\right] \ge R\right) \le e^{-\frac{\nu\lambda_1}{8K_B}R},
\end{equation*}
for all $T\ge0, R>0$.
\end{lemma}
\begin{proof}
 Let us start from the estimate 
\begin{equation*}
\|u(t)\|^2_H +  \nu\int_0^t \|u(s)\|^2_V\, {\rm d}s  \le \|u_0\|_H^2+\left( K_B+ \frac{1}{\nu} \|f\|^2_{V^*}\right)t 
 +2 \int_0^t\langle u(s), G(u(s))\, {\rm d}W(s)\rangle,
\end{equation*}
which holds $\mathbb{P}$-a.s., for all $t\ge0$.
Set 
\begin{equation}\label{la-mart}
M(t):= 2 \int_0^t\langle u(s), G(u(s))\, {\rm d}W(s)\rangle, \quad t \ge 0.
\end{equation}
This is a martingale process with quadratic variation
\begin{equation}
[M](t) \le 
4K_B\int_0^t \|u(s)\|_H^2\,{\rm d}s
\underset{\text{by  }\eqref{Poincarè}} {\le}
\frac{4K_B}{\lambda_1} \int_0^t \|u(s)\|_V^2\,{\rm d}s, \qquad t \ge 0.
\end{equation}
Hence
\begin{align*}
\|u(t)\|^2_H +\frac{\nu}{2}\int_0^t \|u(s)\|^2_V\, {\rm d}s  -  \|u_0\|_H^2- \left(K_B+ \frac{1}{\nu\lambda_1} \|f\|^2_H \right) t 
\le 
&M(t) -\frac{\nu}{2}\int_0^t \|u(s)\|^2_V\, {\rm d}s
\\
\le 
&M(t) - \frac{\nu\lambda_1}{8K_B}[M](t).
\end{align*}
We now recall the 
the exponential martingale inequality (see e.g. \cite[Chapter 6, Section 8]{Kry} and \cite[Proposition 3.1]{GHnotes})
\begin{equation*}
\mathbb{P}\left( \sup_{t \ge 0} \left[M(t)-\gamma [M](t)\right] \ge R\right) \le e^{-\gamma R},  
\end{equation*}
which holds for any $R$ and $\gamma>0$. Taking $\gamma= \frac{\nu\lambda_1}{8K_B}$, 
from the previous estimates the thesis follows.
\end{proof}

\begin{lemma}
\label{lem_mul_3}
Let Assumptions \ref{assumption-stochastic1} and \ref{assumption-stochastic2}(S) be  in force.
Let $u$ denote the solution to the Navier-Stokes equation \eqref{NS_abs_mult}. Then, 
for an arbitrary $q \in (2, + \infty)$, it holds
\begin{equation*}
\mathbb{P}\left( \sup_{t\ge T}\left[ \|u(t)\|^2_H + \nu \int_0^t \|u(s)\|^2_V \, {\rm d}s -\|u_0\|^2_H -C_S\left(t+1\right)\right] \ge R\right) \le \frac{C(1+ \|u_0\|_H^{2q})}{(T+R)^{\frac q2 -1}},
\end{equation*}
for all $T\ge0, R>0$. Here $C_S$ is the constant defined  in \eqref{C_S} and $C$ is a positive constant depending on $K_S, \tilde K_S, \nu, \lambda_1, q, \gamma, \|f\|_{V^*}$, independent of $T,R$ and $u_0$.
\end{lemma}
\begin{proof}
We start with the estimate 
\begin{equation}
\label{mar_t}
\|u(t)\|^2_H + \nu\int_0^t \|u(s)\|^2_V\, {\rm d}s -  \|u_0\|_H^2 
 - \left(K_S + 
 (1-\gamma)\left(\frac{2\gamma}{\lambda_1 \nu}\right)^{\frac{\gamma}{1-\gamma}} \widetilde{K_S}^{\frac{1}{1-\gamma}}
 +
\frac{2}{\nu}\|f\|^2_{V^*} \right)t  \le M(t),
\end{equation}
which hold $\mathbb{P}$-a.s. for any $t\ge0$, where  $M$ is the martingale defined in \eqref{la-mart}.
We estimate its  quadratic variation
\begin{align}
\label{qua_var_S}
[M](t)
&\le 4  \int_0^t \|u(s)\|^2_H\|G(u(s))\|^2_{L_{HS}(U,H)}\, {\rm d}s
\le C_{K_S, \widetilde{K}_S, \gamma}\int_0^t \left(1+ \|u(s)\|^4_H \right)\, {\rm d}s,
\end{align}
where $C_{K_S, \widetilde{K}_S, \gamma}$ is a positive constant.
We now subtract $t+2$ to both sides of \eqref{mar_t}, appealing to \cite[Lemma B.7]{DPSZ25}, bearing in mind \eqref{C_S},
\begin{footnote}{We point out that in our previous work (see [19, Proposition A.4]) the definition of $C_S$	
  featured a minimum. This was a typographical error: one should either use a maximum or, as in the present setting, introduce a multiplicative factor $2$. The same remark applies to the constant $C_L$. We also remark that the suitable bound we employ depends on the type of estimate we aim to obtain.}\end{footnote}
we obtain, for any $T \ge 0$,
\begin{align*}
    \mathbb{P}\left(\sup_{t \ge T}\left(\|u(t)\|^2_H + \nu\int_0^t \|u(s)\|^2_V\, {\rm d}s -  \|u_0\|_H^2 
 - C_S(t+1) \right)>R \right) 
 &\le  \mathbb{P}\left(\sup_{t\ge T}\left( M(t)-t-2\right) \right) 
 \\
 &\le c_q  \sum_{m \ge  \lfloor T\rfloor}\frac{\mathbb{E}\left[[M](m+1)^{\frac q2}\right]} {(R+m+2)^q}.
\end{align*}
Estimate \eqref{qua_var_S} and Lemma \ref{lem_mul_q}(ii) yield
\begin{align}
\label{est_quadr_var}
\mathbb{E}\left[ [M](t)^{\frac q2}\right]
&\lesssim_{q, K_S, \widetilde {K}_S, \gamma} \mathbb{E}\left[\left(\int_0^t  \left(1+\|u(s)\|^{4}_H\right)\, {\rm d}s\right)^{\frac q2}\right]
\\\notag
&\lesssim_{q, K_S, \widetilde {K}_S, \gamma}t^{\frac{q-2}{2}} \mathbb{E}\left[\int_0^t  \left(1+\|u(s)\|^{2q}_H\right)\, {\rm d}s\right]
\le C(t+1)^{\frac{q}{2}}\left(1+\|u_0\|^{2q}_H \right),
\end{align} 
where $C=C(K_S, \tilde K_S, \nu, \lambda_1, q, \gamma, \|f\|_{V^*})$ is a positive constant. Thus, we have 
\begin{align*}
\mathbb{P}\left(\sup_{t \ge T}\left[ M(t)-t-2\right] \ge R \right) 
&\le \sum_{m \ge  \lfloor T\rfloor}\frac{\mathbb{E}\left[[M](m+1)^{\frac q2}\right]} {(R+m+2)^q}
\\
&\le C(1+ \|u_0\|_H^{2q})\sum_{m \ge  \lfloor T\rfloor}\frac{(m+2)^{\frac q2}} {(R+m+2)^q}
\le C(1+ \|u_0\|_H^{2q})\sum_{m \ge  \lfloor T\rfloor}\frac{1} {(R+m+2)^{\frac q2}}.
\end{align*}
The latter series is convergent when $q > 2$ and thus we obtain
\[
\mathbb{P}\left(\sup_{t \ge T}\left[M(t)-t-2\right] \ge R \right)
\lesssim
  \frac{C(1+ \|u_0\|_H^{2q})}{(T+R)^{\frac q2-1}}
\]
and the thesis follows.
\end{proof}

\begin{lemma}
\label{lem_mul_4}
Let Assumptions \ref{assumption-stochastic1} and \ref{assumption-stochastic2}(L) be in force. Let $\nu > \frac{2\widetilde K_L}{\lambda_1}$.
Let $u$ denote the solution to the Navier-Stokes equation \eqref{NS_abs_mult}. Then, 
for an arbitrary $q\in \left(2, 1+ \frac{\nu\lambda_1}{2\widetilde K_L}\right)$, it holds
\begin{equation*}
\mathbb{P}\left( \sup_{t\ge T}\left[ \|u(t)\|^2_H + \left(\nu-\frac{\widetilde{K}_L}{\lambda_1}\right) \int_0^t \|u(s)\|^2_V \, {\rm d}s -\|u_0\|^2_H -C_L\left(t+1\right)\right] \ge R\right) \le \frac{C(1+ \|u_0\|_H^{2q})}{(T+R)^{\frac q2 -1}},
\end{equation*}
for all $T\ge0, R>0$. Here $C_L$ is the constant defined  in \eqref{C_L} and $C$ is a positive constant depending on $K_L,  \nu, \lambda_1, q, \gamma, \|f\|_{V^*}$, independent of $T,R$ and $u_0$.
\end{lemma}
\begin{proof}
    We skip the proof since it is similar to the proof of Lemma \ref{lem_mul_3}. Notice that the conditions $\nu >\frac{2\widetilde{K}_L}{\lambda1}$ and $q<1+ \frac{\nu\lambda_1}{2\widetilde{K}_L}$ are required in order to appeal to Lemma \ref{lem_mul_q}(iii).
\end{proof}

\section{Apriori estimates: the additive noise case}
\label{sec:appB}

Here we collect the basic results in the additive noise case. No proof is given, since these results are particular cases of those obtained with a multiplicative bounded noise in the previous sections, considering $K_B=\|G\|^2_{L_{HS}(U;H)}$ and $L=0$.
\subsection{Estimates in expected value}
\begin{lemma}
\label{lem_add_1}
Assume $G\in L_{HS}(U;H)$ and  $u$ denote the solution to the Navier-Stokes equation \eqref{2D-NS-additive-abs}.
 Then, for every $t \ge 0$
\begin{equation}
\label{E_est_add}
\mathbb{E}[\|u(t)\|^2_H]+ \nu\mathbb{E}\int_0^t \|u(s)\|^2_V\, {\rm d}s  \le  \|u_0\|_H^2+\left( \|G\|^2_{L_{HS}(U,H)}+ \frac{1}{\nu} \|f\|^2_{V^*}\right)t
\end{equation}
and 
\begin{equation}
\label{E_est_add2}
\mathbb{E}\left[\|u(t)\|^2_H\right]  \le  \|u_0\|_H^2 e^{-\nu \lambda_1t}+\frac{1}{\nu \lambda_1}\left( \|G\|^2_{L_{HS}(U,H)}+ \frac{1}{\nu} \|f\|^2_{V^*}\right).
\end{equation}
\end{lemma}

\begin{lemma}
\label{q_add_u_U}
Assume $G\in L_{HS}(U;H)$. 
\begin{itemize}
    \item [(i)]Let $u$ denote the solution to  equation \eqref{2D-NS-additive-abs}. Then for all $q\in [2, +\infty)$,
\begin{equation*}
\mathbb{E}\left[\|u(t)\|^q_H\right]  \le \|u_0\|^q_H + C,
\end{equation*}
where $C$ is a positive constant depending on $q,\nu$,$\lambda_1$, $\|f\|_{V^*}$ and $\|G\|^2_{L_{HS}(U,H)}$.
\item [(ii)]
Let $U$ denote the solution to  equation \eqref{data-assi-NS-equation-abs}. Then for all $q\in [2, +\infty)$,
\begin{equation*}
\mathbb{E}\left[\|U(t)\|^q_H\right]  \le C(1+ \|u_0\|^q_H + \|U_0\|^q_H), 
\end{equation*}
where $C$ is a positive constant depending on $q,\nu$, $\lambda_1$, $\|f\|_{V^*}$ and $\|G\|^2_{L_{HS}(U,H)}$.
\end{itemize}
\end{lemma}

\subsection{Estimates in probability}
\begin{lemma}
\label{lem_add_2}
Assume $G\in L_{HS}(U;H)$ and let $u$ denote the solution to the Navier-Stokes equation \eqref{2D-NS-additive-abs}. 
Then
\begin{equation}
\label{P_est_add}
\mathbb{P}\left( \sup_{t\ge T}\left[ \|u(t)\|^2_H + \frac{\nu}{2} \int_0^t \|u(s)\|^2_V \, {\rm d}s -\|u_0\|^2_H -\left( \|G\|^2_{L_{HS}(U,H)}+ \frac{\|f\|^2_{V^*}}{\nu }\right)t\right] \ge R\right) \le e^{-\frac{\nu\lambda_1}{8\|G\|^2_{L_{HS}(U,H)}}R},
\end{equation}
for all $T\ge 0$, $R>0$.
\end{lemma}

This implies that
\[
\mathbb{P}\left( \sup_{t\ge T}\left[ \frac{\nu}{2} \int_0^t \|u(s)\|^2_V \, {\rm d}s -\|u_0\|^2_H -\left( \|G\|^2_{L_{HS}(U,H)}+ \frac{\|f\|^2_{V^*}}{\nu }\right)t\right] \ge R\right) \le e^{-\frac{\nu\lambda_1}{8\|G\|^2_{L_{HS}(U,H)}}R},
\]
equivalently
\[
\mathbb{P}\left(\frac 1 {\nu t} \int_0^t \|u(s)\|^2_V \, {\rm d}s\le \frac 2{\nu^2} \left( \|G\|^2_{L_{HS}(U,H)}+ \frac{\|f\|^2_{V^*}}{\nu }\right)  + \frac 2{\nu^2 t}(\|u_0\|^2_H + R)\text{ for all } t\ge T\right) \ge 1- e^{-\frac{\nu\lambda_1}{8\|G\|^2_{L_{HS}(U,H)}}R}.
\]
Therefore
\begin{equation}\label{stima-certa}
\mathbb{P}\left(\limsup_{t\to+\infty} \frac 1 {\nu t} \int_0^t \|u(s)\|^2_V \, {\rm d}s\le \frac 2{\nu^2} \Big( \|G\|^2_{L_{HS}(U,H)}+ \frac{\|f\|^2_{V^*}}{\nu }\Big) \right) =1.
\end{equation}

\section{A different proof of Theorem \ref{pathwise_data_ass}}
\label{sec:appC}
In this appendix we present an alternative proof of Theorem \ref{pathwise_data_ass}. The argument relies on the existence of a unique ergodic invariant measure and on suitable moment bounds with respect to it. 
Notice that the assumptions on $G$ in the previous sections are not enough to obtain the  uniqueness of the invariant measure. For simplicity, here we simply assume that such a unique invariant measure exists and refer the reader to the existing literature for the precise conditions ensuring its existence, uniqueness, and ergodicity (see, for instance, \cite{FZ25} and the references therein). 
The next step consists in getting  some moment bounds with respect to the unique invariant  measure.

We denote by $u(t;x)$ the solution to equation \eqref{2D-NS-additive-abs} evaluated at time $t>0$,  started at time $0$  from $x$.
 We introduce the operator $P_t$  for each fixed time  $t>0$ as 
\begin{equation}
    \label{sem}
(P_t\phi)(x)=\mathbb E[\phi(u(t;x))], \qquad  \phi\in \mathcal{B}_b(H),
\end{equation}
where $\mathcal{B}_b(H)$ is the  space of all bounded Borel functions $\phi: H \to \mathbb R$.
This defines a  Markov semigroup acting on $\mathcal{B}_b(H)$. 
We recall that a  probability measure $\gamma$, defined on the Borel subsets of $H$, is {\em invariant} if
 for  any $t\ge 0$,
\begin{equation}\label{def-mu-inv}
\int_H P_t \phi \: {\rm d}\gamma =\int_H \phi \:{\rm d}\gamma \hspace{1cm} \forall \phi \in \mathcal{B}_b(H).
\end{equation}

\begin{lemma}
\label{ergodic}
Assume that equation \eqref{2D-NS-additive-abs} admits a unique ergodic invariant measure $\gamma$.
Then
\begin{equation}
    \label{est_inv_meas_V}
    \int_H \|x\|^2_V \, {\rm d}\gamma(x) \le \frac{2}{\nu}\left(  \|G\|^2_{L_{HS}(U,H)}+ \frac{1}{\nu} \|f\|^2_{V^*}\right).
    \end{equation}
\end{lemma}
\begin{proof}
To prove estimate \eqref{est_inv_meas_V} we proceed in a classical way. 
We consider the identity \eqref{def-mu-inv} and thanks to the definition  \eqref{sem} of the operator $P_t$, we use the previous mean estimates on $\|u(t)\|_H^2$ and $\int_0^t \|u(s)\|_V^2{\rm d}s$.

We start by proving that the function $\Psi:H \rightarrow \mathbb{R}^+$, $\Psi(x):=\|x\|^2_H$ belongs to $L^1(\gamma)$ and 
\begin{equation}
    \label{est_inv_meas_H}
    \int_H \|x\|^2_H \, {\rm d}\gamma(x) \le  \frac{1}{\nu \lambda_1}\left( \|G\|^2_{L_{HS}(U,H)}+ \frac{1}{\nu } \|f\|^2_{V^*}\right).
\end{equation}
We consider the bounded map $\Psi_k : H \rightarrow \mathbb{R}_+$ defined as 
\[
\Psi_k (x):= 
\begin{cases}
    \|x\|^2_H & \text{if} \ \|x\|^2_H \le k,
    \\
    k & \text{otherwise}.
\end{cases}
\]
By the invariance of the measure $\gamma$ and the boundedness of $\Psi_k$ we get 
\[
\int_H \Psi_k(x)\, {\rm d}\gamma(x) = \int_HP_s\Psi_k(x)\, {\rm d}\gamma(x), \qquad \forall \ s \ge 0.
\]
Moreover, 
\[
P_s \Psi_k(x):=\mathbb{E}\left[ \Psi_k (u(s;x))\right]\le \mathbb{E}\left[ \|u(s;x)\|^2_H\right].
\]
From Lemma \ref{lem_add_1} we get an estimate for $\mathbb{E}\left[ \|u(s;x)\|^2_H\right]$ (see \eqref{E_est_add2}). Letting $s \rightarrow + \infty$ the exponential term in the r.h.s. of \eqref{E_est_add2} vanishes, so we obtain the estimate
\[
\limsup_{s \rightarrow + \infty } P_s \Psi_k(x) \le \frac{1}{\nu \lambda_1}\left( \|G\|^2_{L_{HS}(U,H)}+ \frac{1}{\nu } \|f\|^2_{V^*}\right).
\]
As a consequence, 
\[
\int_H \Psi_k(x)\, {\rm d}\gamma(x) \le \frac{1}{\nu \lambda_1}\left( \|G\|^2_{L_{HS}(U,H)}+ \frac{1}{\nu} \|f\|^2_{V^*}\right).
\]
Since $\Psi_k$ converges pointwise and monotonically from below to $\Psi$, from the Monotone Convergence Theorem we infer \eqref{est_inv_meas_H}.

Let us now come to the proof of \eqref{est_inv_meas_V}. We consider the bounded map 
$\phi_k : H \rightarrow \mathbb{R}_+$  defined as 
\[
\phi_k (x):= 
\begin{cases}
    \|x\|^2_V & \text{if} \ \|x\|^2_V \le k,
    \\
    k & \text{otherwise}.
\end{cases}
\]
This is a truncation, approximating  the function $\phi:  H \to  \mathbb{R}_+\cup\{+\infty\}$, defined as 
$\phi(x)= \|x\|^2_V$.
It obviously holds 
\[
\int_H \phi_k(x)\, {\rm d}\gamma(x)= \nu \lambda_1\int_0^{\frac{1}{\nu\lambda_1}} \int_H \phi_k(x)\, {\rm d}\gamma(x)\, {\rm d}s.
\]
On the other hand, the invariance of $\gamma$ and the boundedness of $\phi_k$ yield 
\[
\int_H \phi_k(x)\, {\rm d}\gamma(x) = \int_H P_s\phi_k(x)\, {\rm d}\gamma(x), \qquad \forall \ s \ge 0.
\]
Thus, by the Fubini-Tonelli Theorem, since 
\[
\phi_k(x)=\|x\|^2_V \wedge k \le \|x\|^2_V, 
\]
thanks to Lemma \ref{lem_add_1} (see \eqref{E_est_add}) we get 
\begin{align*}
    \int_H \phi_k(x)\, {\rm d}\gamma(x) 
    & = \nu\lambda_1 \int_0^{\frac{1}{\nu\lambda_1}} \int_H \phi_k(x)\, {\rm d}\gamma(x)\, {\rm d}s = \nu\lambda_1 \int_0^{{\frac{1}{\nu\lambda_1}}} \int_H P_s\phi_k(x)\, {\rm d}\gamma(x)\, {\rm d}s
    \\
    & =\nu\lambda_1\int_0^{{\frac{1}{\nu\lambda_1}}} \int_H \mathbb{E}\left[ \phi_k(u(s;x))\right]\, {\rm d}\gamma(x)\, {\rm d}s 
    = \nu\lambda_1\int_H \mathbb{E}\int_0^{{\frac{1}{\nu\lambda_1}}} \phi_k(u(s;x))\, {\rm d}s \, {\rm d}\gamma(x)
    \\
    &\le \lambda_1 \int_H \nu\mathbb{E}\int_0^{{\frac{1}{\nu\lambda_1}}} \left[\|u(s;x)\|^2_V \right]\, {\rm d}s \, {\rm d}\gamma(x)
    \\
    &\le \lambda_1\left( \frac{1}{\nu\lambda_1}\left( \|G\|^2_{L_{HS}(U,H)}+ \frac{1}{\nu} \|f\|^2_{V^*}\right) +\int_H \|x\|_H^2\, {\rm d}\gamma(x)\right).
\end{align*}
Estimate \eqref{est_inv_meas_H} thus yields 
\[
 \int_H \phi_k(x)\, {\rm d}\gamma(x) \le \frac{2}{\nu}\left(  \|G\|^2_{L_{HS}(U,H)}+ \frac{1}{\nu} \|f\|^2_{V^*}\right).
\]
Since $\phi_k$ converges pointwise and monotonically from below to $\Phi$,
estimate \eqref{est_inv_meas_V} follows. This concludes the proof.
\end{proof}

\begin{proof}[Alternative proof of Theorem \ref{pathwise_data_ass} (under more restrictive assumptions on $G$)]
We work under the assumptions on $G$ ensuring the existence of a unique ergodic invariant measure.
The proof relies on the Birkhoff Ergodic Theorem (see, e.g., \cite{sinai}). It states that, if $\gamma$ is an ergodic invariant measure for the Navier-Stokes equation \eqref{2D-NS-additive-abs}), then for any $\phi\in L^1(\gamma)$ and 
 for $\gamma$-a.e. initial velocity, it holds
\begin{equation}\label{Birk-f}
\lim_{t\to +\infty}\frac 1t \int_0^t \phi(u(s))\, {\rm d}s=\int_H \phi\,{\rm d}\gamma \qquad
\mathbb P-a.s.
\end{equation}
On the other hand, from Lemma \ref{ergodic} we know that $\phi(x)=\|x\|^2_V \in L^1(\gamma)$ and we have the bound \eqref{est_inv_meas_V}. Therefore, from the Birkhoff Ergodic Theorem we can infer
\[
\lim_{t\to +\infty}\frac{1}{\nu t} \int_0^t \|u(s)\|^2_V\, {\rm d}s=\frac{1}{\nu}\int_H \|x\|^2_V\,{\rm d}\gamma(x) \le \frac{2}{\nu^2}\left(  \|G\|^2_{L_{HS}(U,H)}+ \frac{1}{\nu} \|f\|^2_{V^*}\right), \qquad
\mathbb P-a.s.
\]
From now on we proceed as in the proof of Theorem \ref{pathwise_data_ass} to conclude.
    
\end{proof}


\end{document}